\documentclass[12pt]{amsart}

\usepackage{epsf,epsfig,amsfonts,amsgen,amsmath,amstext,amsbsy,amsopn,amsthm
}
\usepackage{amsmath,times,mathptmx}
\usepackage{amsmath}%
\usepackage{amsfonts,amsthm,amssymb}
\usepackage{amsfonts}
\usepackage{graphics}
\usepackage{latexsym,bm}
\usepackage{amsfonts,amsthm,amssymb,bbding}
\usepackage{indentfirst}
\usepackage{graphicx}
\usepackage{color}
\usepackage[colorlinks=true,anchorcolor=blue,filecolor=blue,linkcolor=blue,urlcolor=blue,citecolor=blue]{hyperref}
\usepackage{float}
\usepackage{tikz}
\usepackage{verbatim}
\usepackage{mathrsfs}
\usetikzlibrary{calc,decorations.pathreplacing}
\evensidemargin=\oddsidemargin\topmargin=-1.5cm

\newtheorem{thm}{Theorem}[section]
\newtheorem{prob}{Problem}

\newtheorem{lem}{Lemma}[section]
\newtheorem{cor}{Corollary}[section]

\newtheorem{pro}{Proposition}[section]

\newtheorem{exam}{Example}

\newtheorem{claim}{Claim}[section]

\newtheorem{definition}{Definition}[section]

\addtocounter{section}{0}

\begin{document}
\title{Minimal spectral radius of graphs with given matching  number}
\author[J. Liu]{Jiaqi Liu}
\author[Z. Lou] {Zhenzhen Lou}
\address{College of Science, University of Shanghai for Science and Technology,
 Shanghai, 200093, P.R. China}
\email{m19821217592@163.com, louzz@usst.edu.cn}
 \author[V. Trevisan]{Vilmar Trevisan}
\address{Instituto de Matemática e Estatística,  Universidade Federal do Rio Grande do Sul, Porto Alegre, Brazil}
\email{trevisan@mat.ufrgs.br}
\date{December 2025}

\begin{abstract}
The Brualdi-Solheid problem asks which graph achieves the extremal (maximum or minimum) spectral radius for a given class of graphs.
This paper addresses the Brualdi-Solheid problem for \( \mathcal{G}_{n,\beta} \), the family of graphs with order \( n \) and matching number \( \beta \), aiming to identify its spectrally minimal graphs i.e., those that minimize the spectral radius \(\rho(G)\).

We introduce the novel concept of ``quasi-adjacency'' relation, developing a unified structural classification framework for trees in \(\mathcal{G}_{n,\beta}\), which clarifies structural properties and provides a constructive method to generate trees with fixed \(\beta\). By showing that all spectrally minimal graphs in \( \mathcal{G}_{n,\beta} \) are trees, we further narrow the search for extremal graphs. Additionally, we apply this framework to the representative cases \(\beta=2,3,4\), obtaining  the minimizers  by explicit structural formulas involving parameters related to \(n\).

 \begin{flushleft}
\end{flushleft}
\end{abstract}

\maketitle

\section{Introduction}\label{se-1}

Since its formulation by Brualdi and Solheid in 1986 \cite{Brualdi1986}, the following problem has emerged as a cornerstone of spectral graph theory.

\begin{prob}(\cite{Brualdi1986})\label{p1}
For a given family of graphs \( \mathbb{G} \), determine $\min \left\{ \rho(G): G \in \mathbb{G}\right\}\), and  \( \max\left\{\rho(G): G\in \mathbb{G}\right\}\), and characterize the graphs that attain these extreme values.
\end{prob}

Over the past decades, researchers have extensively explored the relationships between the spectral radius and other graph invariants, leading to significant advancements that can be found in the vast available literature. 

For the independence number \( \alpha \), Ji and Lu \cite{JiLu} determined the graphs with the maximal spectral radius among all trees on \( n \) vertices with independence number \( \alpha \), where \( \lceil \tfrac{n}{2} \rceil \le \alpha \le n-1 \). Du and Shi (2013) \cite{DuShi2013} showed that blowing up each vertex of a path of order \( \alpha \) into a clique of order \( k \) minimizes the spectral radius among connected graphs of order \( k\alpha \) when \( \alpha = 3, 4 \), and conjectured this holds for all \( \alpha \in \mathbb{N} \). Xu et al. (2009) \cite{Xu2009} characterized the minimizers for \( \alpha \in \{1, 2, \lfloor \frac{n}{2} \rfloor, \lfloor \frac{n}{2} \rfloor + 1, n - 3, n - 2, n - 1\} \), while Lou and Guo (2022) \cite{LouGuo2022} proved that minimizers for \( \alpha \geq \lfloor \frac{n}{2} \rfloor \) must be trees and characterized those for \( \alpha = n - 4 \), for \( n \geq 19 \).

For the diameter \( d \), Guo and Zhou \cite{GuoZhou2020} determined the unique tree with the maximum spectral radius among all trees of a given diameter. Kumar et al. (2025) \cite{Kumar2025} and Wang et al. (2023) \cite{WangShan2023} characterized extremal trees and unicyclic/bicyclic graphs, respectively, while He and Lu (2022) \cite{HeLu2022} determined the maximum spectral radius of trees with fixed odd diameter. Further related results can be found in \cite{ChoiPark, GuoZhou2020, HouLiu2021, JiLu, li2023two, SimiBelardo2010}.

Regarding the domination number \( \gamma \), Liu et al. \cite{LiuLiXie2023} and Hu et al. \cite{HuLou2023} identified minimizers for specific values of \( \gamma \), while Xing and Zhou \cite{XingZhou2015} investigated graphs with maximal Laplacian and signless Laplacian spectral radii. Regarding the girth, Lou et al. \cite{lou2024} determined the graphs maximizing the spectral radius among those with odd girth \( k \) and odd number of edges \( m \). Patra and Sahoo \cite{Patra2013} studied the minimization of the Laplacian spectral radius of unicyclic graphs with fixed girth. Further related results can be found in \cite{ChoiPark, HouLiu2021}.

Despite these advancements, a critical gap remains: the Brualdi-Hoffman problem has not been systematically addressed for graphs with a fixed matching number \( \beta \). Early work by Hou and Li (2002) \cite{HouLi2002} provided upper bounds on the spectral radius of trees with fixed \( n \) and \( \beta \), and Feng et al. (2007) \cite{FengYu2007} characterized maximizers for \( \beta \)-constrained graphs. Li and Chang (2014) \cite{LiChang2014} and Sun et al. (2017) \cite{Sun2017} studied minimal Laplacian and spectral radii of trees for specific \( n = k\beta + 1 \), and Zhai et al. (2022) \cite{ZhaiXue2022} determined maximal \( Q \)-spectral radii for graphs with fixed size and \( \beta \) (see also \cite{Chang2003, Cioaba2005, LinGuo2007}). However, the structural characterization of spectrally minimal graphs for general \( n \) and \( \beta \)—a fundamental aspect of the Brualdi-Hoffman problem—remains unresolved. To fill this gap, we focus on the following core problem:

\begin{prob}\label{p2}
For the family \( \mathcal{G}_{n,\beta} \) of graphs with order \( n \) and matching number \( \beta \), what are the spectrally minimal graphs (i.e., those minimizing \( \rho(G) \))?
\end{prob}

We recall that for a graph \( G = (V, E) \), a \emph{matching} \( M \subseteq E \) is a set of edges where no two distinct edges are adjacent; a maximum matching is one with the largest possible cardinality, called the \emph{matching number} and denoted by \( \beta \) of \( G \). A perfect matching is a maximum matching that covers every vertex in \( G \). Additional key concepts include quasi-pendant vertices (the unique neighbors of pendant vertices), blocks (maximal connected subgraphs without cut vertices), and the spectral radius \( \rho(G) \) (the largest eigenvalue of the adjacency matrix \( A(G) \) of \( G \)).

In this work, we study Problem~\ref{p2}. The  main contributions are as follows. First, we establish a critical structural result: using a spectral radius comparison model, proof by contradiction, and graph structure simplification techniques, we prove that all spectrally minimal graphs in \( \mathcal{G}_{n,\beta} \) must be trees. This result drastically narrows the search space for minimizers from general graphs to trees, a theoretical breakthrough that simplifies subsequent analysis.

Second, we introduce the new concept of ``quasi-adjacency'' relation to develop a unified structural classification framework for trees in \( \mathcal{G}_{n,\beta} \). This framework clarifies the structural properties of such trees and provides a constructive method to generate trees with a fixed \( \beta \), resolving the lack of systematic structural understanding in prior work.

Third, we apply this framework to the specific cases \( \beta = 2, 3, 4 \) (representative values that generalize to larger \( \beta \)) and classify trees in \( \mathcal{G}_{n,\beta} \) into distinct subclasses. Through comprehensive spectral radius analysis—including constructing characteristic equations of adjacency matrices to rule out subclasses, and deriving strict spectral radius inequalities to exclude others—we characterize the spectrally minimal graphs. 

For completeness, we recall that if \( G \in \mathcal{G}_{n,\beta} \) and \( \beta = 1 \), then \( G \) is either a triangle (\( n = 3 \)) or the star \( K_{1,n-1} \). It is clear that \( K_{1,n-1} \) achieves the minimal spectral radius. For \( \beta = 2 \), we derive the minimizer among all graphs in \( \mathcal{G}_{n,2} \) as follows.

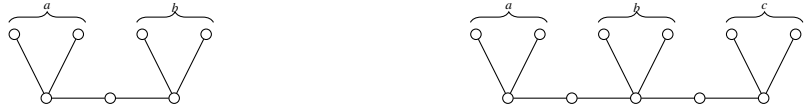
\begin{figure}
   \centering
    \begin{minipage}[b]{0.45\linewidth}
    \centering
\begin{tikzpicture}
[scale=.85,auto=left,every node/.style={circle,scale=.35}]
\foreach \i in {1,2,3}{
     \node[draw,circle] (\i) at (\i, 0){};}
\foreach \from/\to in {1/2,2/3}{
\draw (\from) -- (\to);}
\node[draw,circle] (a) at (.5, 1){};
\node[draw,circle] (b) at (1.5, 1){};
\draw (a)--(1)--(b);
\draw [decorate,decoration={brace,mirror,amplitude=4pt},xshift=0.4pt,yshift=-0.4pt]
(1.6,1.2) -- (.4,1.2) node [black,midway,yshift=1cm]{\large $a$};

\node[draw,circle] (a) at (2.5, 1){};
\node[draw,circle] (b) at (3.5, 1){};
\draw (a)--(3)--(b);
\draw [decorate,decoration={brace,mirror,amplitude=4pt},xshift=0.4pt,yshift=-0.4pt]
(3.6,1.2) -- (2.4,1.2) node [black,midway,yshift=1cm]{\large $b$};

\end{tikzpicture}
 \end{minipage}
    \hfill
  \begin{minipage}[b]{0.45\linewidth}
\begin{tikzpicture}
[scale=.85,auto=left,every node/.style={circle,scale=.35}]
\foreach \i in {1,2,3,4,5}{
     \node[draw,circle] (\i) at (\i, 0){};}
\foreach \from/\to in {1/2,2/3,3/4,4/5}{
\draw (\from) -- (\to);}

\node[draw,circle] (a) at (.5, 1){};
\node[draw,circle] (b) at (1.5, ){};
\draw (a)--(1)--(b);
\draw [decorate,decoration={brace,mirror,amplitude=4pt},xshift=0.4pt,yshift=-0.4pt]
(1.6,1.2) -- (.4,1.2) node [black,midway,yshift=1cm]{\large $a$};

\node[draw,circle] (a) at (2.5, 1){};
\node[draw,circle] (b) at (3.5, 1){};
\draw (a)--(3)--(b);
\draw [decorate,decoration={brace,mirror,amplitude=4pt},xshift=0.4pt,yshift=-0.4pt]
(3.6,1.2) -- (2.4,1.2) node [black,midway,yshift=1cm]{\large $b$};

\node[draw,circle] (a) at (4.5, 1){};
\node[draw,circle] (b) at (5.5, 1){};
\draw (a)--(5)--(b);
\draw [decorate,decoration={brace,mirror,amplitude=4pt},xshift=0.4pt,yshift=-0.4pt]
(5.6,1.2) -- (4.4,1.2) node [black,midway,yshift=1cm]{\large $c$};

\end{tikzpicture}
\end{minipage}
\caption{The graphs $T_{2}(a,b)$ (left) and $T_{3}(a,b,c)$ (right) .}\label{fig-2}
\end{figure}

\begin{thm}\label{thm-1.1}
Let \( T_2^* \) be the graph with the minimal spectral radius among all graphs in \( \mathcal{G}_{n,2} \),  with $n\geq 4$. Then \( T_2^* \cong T_2\left( \left\lfloor \tfrac{n-3}{2} \right\rfloor, \left\lceil \tfrac{n-3}{2} \right\rceil \right) \)(see Fig. \ref{fig-2}, left).
\end{thm}

For \( \beta = 3 \), we derive the minimizer among all graphs in \( \mathcal{G}_{n,3} \) as follows.

\begin{thm}\label{matching3-min}
Let \( T_3^* \) be an extremal graph with the minimum spectral radius over all graphs in \( T_3(a,b,c) \)(see Fig.\ref{fig-2}, right), where \( a + b + c = n - 5 \) (\( n \geq 11 \)). Then
\[
T_3^* \cong 
\begin{cases}
T_3\!\left( s+1, \, s-2, \, s+1 \right), & \text{if } n-5=3s, \\
T_3\!\left( s+1, \, s-1, \, s+1 \right), & \text{if } n-5=3s+1, \\
T_3\!\left( s+1, \, s, \, s+1 \right), & \text{if } n-5=3s+2.
\end{cases}
\]
\end{thm}

For \( \beta = 4 \), we derive the minimizer among all graphs in \( \mathcal{G}_{n,4} \) as follows.

\begin{thm}\label{thm-1.3}
Let \( n \geq 19 \), \( T_4^* \in \mathcal{G}_{n,4} \) be a graph with minimum spectral radius, and let \( s = \left\lfloor \frac{n - 7}{4} \right\rfloor \). Then
\[
T_4^* \cong \begin{cases} 
\begin{cases} K_6(s + 1, s - 1, s - 1, s + 1) \quad \text{and} \\ K_{10}(s + 1, s -3, s + 1, s +1) \end{cases} & \text{if } n - 7 = 4s, \\
K_{10}(s + 1, s - 2, s + 1, s + 1) & \text{if } n - 7 = 4s + 1, \\
\begin{cases} K_6(s + 1, s, s, s + 1), \\ K_6(s + 2, s - 1, s, s + 1) \quad \text{and} \\ K_6(s + 2, s - 1, s - 1, s + 2) \end{cases} & \text{if } n - 7 = 4s + 2, \\
K_{10}(s + 2, s - 3, s + 2, s + 2) & \text{if } n - 7 = 4s + 3,
\end{cases}
\]
where $K_6(a,b,c,d)$ and $K_{10}(a,b,c,d)$ are shown in Fig.\ref{fig-1.1}.
\end{thm}

\begin{figure}[h!]
    \centering 
\begin{minipage}[l]{.45\linewidth}
	\centering
\begin{tikzpicture}     
[scale=.85,auto=left,every node/.style={circle,scale=.35}]
 \vspace*{-2cm}
\foreach \i in {1,2,3,4,5,6,7}{
     \node[draw,circle] (\i) at (\i, 3){};}
\foreach \from/\to in {1/2,2/3,3/4,4/5,5/6,6/7}{
\draw (\from) -- (\to);}
\node[draw,circle] (a) at (.5, 4){};
\node[draw,circle] (b) at (1.5, 4){};
\draw (a)--(1)--(b);
\draw [decorate,decoration={brace,mirror,amplitude=4pt},xshift=0.4pt,yshift=-0.4pt]
(1.6,4.2) -- (.4,4.2) node [black,midway,yshift=1cm]{\Large $a$};

\node[draw,circle] (a) at (2.5, 4){};
\node[draw,circle] (b) at (3.5, 4){};
\draw (a)--(3)--(b);
\draw [decorate,decoration={brace,mirror,amplitude=4pt},xshift=0.4pt,yshift=-0.4pt]
(3.6,4.2) -- (2.4,4.2) node [black,midway,yshift=1cm]{\Large $b$};

\node[draw,circle] (a) at (4.5, 4){};
\node[draw,circle] (b) at (5.5, 4){};
\draw (a)--(5)--(b);
\draw [decorate,decoration={brace,mirror,amplitude=4pt},xshift=0.4pt,yshift=-0.4pt]
(5.6,4.2) -- (4.4,4.2) node [black,midway,yshift=1cm]{\Large $c$};

\node[draw,circle] (a) at (6.5, 4){};
\node[draw,circle] (b) at (7.5, 4){};
\draw (a)--(7)--(b);
\draw [decorate,decoration={brace,mirror,amplitude=4pt},xshift=0.4pt,yshift=-0.4pt]
(7.6,4.2) -- (6.4,4.2) node [black,midway,yshift=1cm]{\Large $d$};
\end{tikzpicture}
\end{minipage}
\begin{minipage}{.45\linewidth}
	\centering
\begin{tikzpicture}
[scale=.85,auto=left,every node/.style={circle,scale=.35}]

\foreach \i in {1,2,3,4,5}{
     \node[draw,circle] (\i) at (\i, 0){};}
\foreach \from/\to in {1/2,2/3,3/4,4/5}{
\draw (\from) -- (\to);}

\node[draw,circle] (a) at (.5, 1){};
\node[draw,circle] (b) at (1.5, 1){};
\draw (a)--(1)--(b);
\draw [decorate,decoration={brace,mirror,amplitude=4pt},xshift=0.4pt,yshift=-0.4pt]
(1.6,1.2) -- (.4,1.2) node [black,midway,yshift=1cm]{\Large $a$};
\node[draw,circle] (a) at (2.5, 1){};
\node[draw,circle] (b) at (3.5, 1){};
\draw (a)--(3)--(b);
\draw [decorate,decoration={brace,mirror,amplitude=4pt},xshift=0.4pt,yshift=-0.4pt]
(3.6,1.2) -- (2.4,1.2) node [black,midway,yshift=1cm]{\Large $b$};

\node[draw,circle] (a) at (4.5, 1){};
\node[draw,circle] (b) at (5.5, 1){};
\draw (a)--(5)--(b);
\draw [decorate,decoration={brace,mirror,amplitude=4pt},xshift=0.4pt,yshift=-0.4pt]
(5.6,1.2) -- (4.4,1.2) node [black,midway,yshift=1cm]{\Large $c$};

\node[draw,circle] (c) at (3, -1){};
\node[draw,circle] (d) at (3, -2){};
\node[draw,circle] (a) at (2.5,-3){};
\node[draw,circle] (b) at (3.5,-3){};
\draw (a)--(d)--(b); \draw (3)--(c)--(d);
\draw [decorate,decoration={brace,mirror,amplitude=4pt},xshift=0.4pt,yshift=-0.4pt]
(2.4,-3.2) -- (3.6,-3.2) node [black,midway,yshift=-1cm]{\Large $d$};
\end{tikzpicture}
\end{minipage}
\caption{The graphs $T_{6}(a,b,c,d)$ (left) and $T_{10}(a,b,c,d)$ (right) .}\label{fig-1.1}
\end{figure}
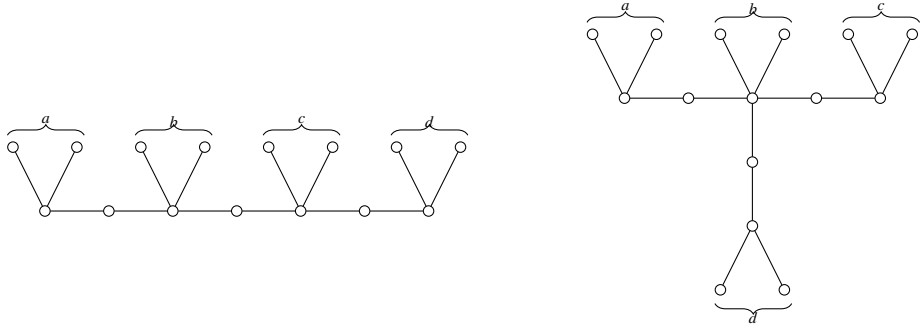

This paper addresses the Brualdi-Hoffman problem for graphs with fixed matching number \( \beta \) by establishing three core results: (1) spectrally minimal graphs in \( \mathcal{G}_{n,\beta} \) are necessarily trees; (2) a unified structural classification framework based on "quasi-adjacency" relations enables systematic analysis of such trees; (3) explicit characterizations of minimizers for \( \beta = 2, 3, 4 \) provide concrete examples of how the framework operates.
These findings advance spectral graph theory by filling a critical gap in the study of extremal graphs with fixed matching numbers.

\section{Quasi-adjacent  graphs}\label{sec-2}
For two vertices $u$ and $v$ in a tree $T$, let \(P_{uv}\) stand for the path  between $u$ and $v$ in $T$.
\begin{definition}\label{quasi-adjacent}
Given a tree $T$ and a control set \(X \subseteq V(T)\), two vertices \(u, v \in X\) are said to be \textit{quasi-adjacent with respect to $X$} 
if there is no other vertex \(w \in X \setminus \{u, v\}\) that lies on the path \(P_{uv}\) in $T$. 
This relation is denoted by \(u \dot{\sim} v\).
\end{definition}

Clearly, the quasi-adjacent relation (denoted \(u \dot{\sim} v\)) is an equivalence relation.
In accordance with the definition of the quasi-adjacent relation between two vertices in a control set $X$ of a tree $T$ (given in Definition \ref{quasi-adjacent}), \textit{the quasi-adjacent graph} of $T$ respect to $X$, is denoted as $T_X$, is constructed as follows: two vertices in $X$ are connected by an edge if and only if they are quasi-adjacent within $X$.
\begin{exam}
Let $T$ be the tree shown in Fig.\ref{eg} on the left, and let $X=\{v_1, v_3, v_4\}$ be a control set of $T$.  
By the definition of quasi-adjacent relation,   $v_1$ and $v_3$ are quasi-adjacent, and $v_3$ and $v_4$ are quasi-adjacent;
	however,  $v_1$ and $v_4$ are not quasi-adjacent because $v_3$ lies on the path $P_{v_1 v_4}$. Clearly, the quasi-adjacent graph of $T$ with respect to $X$ is shown in Fig.\ref{eg}  (right).
\end{exam}	
\vspace{-0.2cm}
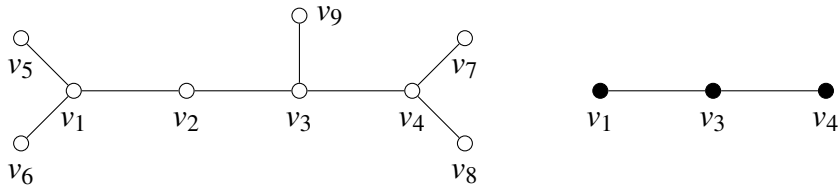
\begin{figure}[h]
	
	\begin{tikzpicture}[scale=1, every node/.style={circle, draw, inner sep=1.5pt}]
		\node (a1) at (0,0) [circle,inner sep=2pt,label=below:{$v_1$}] {};
		\node (a2) at (1.5,0) [circle,inner sep=2pt,label=below:{$v_2$}] {};
		\node (a3) at (3,0) [circle,inner sep=2pt,label=below:{$v_3$}] {};
		\node (a4) at (4.5,0) [circle,inner sep=2pt,label=below:{$v_4$}] {};
		
		\node (l1) at (-0.7,0.7) [circle,inner sep=2pt,label=below:{$v_5$}] {};
		\node (l2) at (-0.7,-0.7) [circle,inner sep=2pt,label=below:{$v_6$}] {};
		
		\node (r1) at (5.2,0.7) [circle,inner sep=2pt,label=below:{$v_7$}] {};
		\node (r2) at (5.2,-0.7) [circle,inner sep=2pt,label=below:{$v_8$}] {};
		
		\node (u1) at (3,1.0) [circle,inner sep=2pt,label=right:{$v_9$}] {};
		
		\draw (l1)--(a1)--(l2);
		\draw (a1)--(a2)--(a3)--(a4);
		\draw (a3)--(u1);
		\draw (a4)--(r1);
		\draw (a4)--(r2);
		
		
		\node (v1) at (7,0) [circle,fill=black,inner sep=2pt,label=below:{$v_1$}] {};
		\node (v3) at (8.5,0) [circle,fill=black,inner sep=2pt,label=below:{$v_3$}] {};
		\node (v4) at (10,0) [circle,fill=black,inner sep=2pt,label=below:{$v_4$}] {};
		\draw (v1)--(v3)--(v4);

	\end{tikzpicture}

	\caption{A graph $T$ (left) and the quasi-adjacent graph of $T$ with respect to $X=\{v_1, v_3, v_4\}$. }\label{eg}
\end{figure}

Note that for any two vertices \(u, v\) of $T$,  there is exactly one path connecting them. 
Consequently, for any tree $T$ with a given control set $X$, there is precisely one corresponding quasi-adjacent graph $T_X$. 
In the subsequent discussion, we will elaborate on some more properties of $T_X$.

Let $\mathcal{G}_{n,\beta}$ denote the family of all $n$-vertex graphs with matching number $\beta$. 
Denote  \( \beta(G)\) the matching number of a graph $G$.
To establish connections between the structure of general graphs and their spanning trees in the context of matching numbers, we first show that the matching number of a connected graph can be preserved in at least one of its spanning trees. This result serves as a bridge for analyzing graph properties through tree structures, which are often more tractable.
\begin{thm}\label{spanning_tree}
For any simple connected graph \( G \in \mathcal{G}_{n,\beta} \), there exists a spanning tree \( T \) of \( G \) with matching number \( \beta \).
\end{thm}

\begin{proof}
Let \( G \) be a simple connected graph in \( \mathcal{G}_{n,\beta} \), and let \( M = \{e_1, \dots, e_\beta\} \) be a maximum matching in \( G \). We construct a spanning tree \( T \) of \( G \) that preserves \( \beta \) as its matching number.

\textbf{Case 1: \( G \) is a tree.}  
Now, \( T = G \) itself is the required spanning tree, as \( M \) remains a maximum matching in \( T \).

\textbf{Case 2: \( G \) contains cycles.}  
Let \( E_M = E(G) \setminus M \). We claim that every cycle \( C \) in \( G \) contains at least one edge \( e \in E_M \). Indeed, if \( C \subseteq M \), then \( C \) would consist of disjoint edges, which is impossible for a cycle. Thus, \( C \cap E_M \neq \emptyset \).

In order to construct \( T \), we iteratively remove edges from cycles as follows:
Select a cycle \( C \) in the current graph $G$.
Remove an arbitrary edge \( e \in C \cap E_M \), yielding \( G' = G - e \).
Since \( e \notin M \),  \( M \) remains a matching in \( G' \) with a size of $\beta$.
The graph \( G' \) stays connected because \( e \) belongs to a cycle.
Repeat until no cycles remain.

The resulting graph \( T \) is a spanning tree of \( G \) because it is connected, acyclic, and contains all vertices of \( G \). Moreover, since \( M \) is preserved throughout the process, we have \( \beta(T) = \beta \).
\end{proof}

For trees \(T \in \mathcal{T}_{n,\beta}\), we define two distances below to study their structural properties.


\textit{Edge-to-edge distance}: For any $e_1, e_2 \in E(T)$,
    \[
    \bar{d}_T(e_1, e_2) = \min_{\substack{v_1 \in e_1 \\ v_2 \in e_2}} d_T(v_1, v_2).
    \]
    
  \textit{Edge-set distance}: For any $E_1, E_2 \subseteq E(T)$,
    \[
    \bar{d}_T(E_1, E_2) = \min_{\substack{e_1 \in E_1 \\ e_2 \in E_2}} \bar{d}_T(e_1, e_2).
    \]

We next focus on the relationship between maximal matchings and induced subgraphs. 
The following lemma establishes the existence of a specific maximal matching with desirable connectivity in the induced subgraph.
\begin{thm}\label{lem-2.7}
For any tree \( T \in \mathcal{T}_{n,\beta} \), there exists a maximal matching 
\[ M = \{e_1 = u_1v_1, e_2 = u_2v_2, \dots, e_\beta = u_\beta v_\beta\} \] 
such that the induced subgraph \( T[U] \) is a tree, where \( U = V(M) = \bigcup_{i=1}^\beta \{u_i, v_i\} \).
\end{thm}

\begin{proof}
Let \( T \in \mathcal{T}_{n,\beta} \) and fix an arbitrary maximal matching 
 \( M = \{e_1 = u_1v_1, \dots, e_\beta = u_\beta v_\beta\} \) 
of \( T \). Let \( U = V(M) \).
If \( T[U] \) is connected, then it must be a tree since \( T \) is acyclic. If \( T[U] \) is disconnected, we iteratively modify \( M \) while preserving maximality until \( T[U] \) becomes connected.

\noindent \textbf{Step 1:} Initialize \( l := 1 \).  
While \( T[U] \) is disconnected, perform the following steps:

\noindent \textbf{Step 2:} Let \( T_l \) be the largest connected component of \( T[U] \).

\noindent \textbf{Step 3:} Select an edge \( e_{t_l} = u_{t_l}v_{t_l} \in M \setminus E(T_l) \) that minimizes the distance \( \bar{d}_T(E(T_l), e_{t_l}) \).  
Let \( u'v' \in E(T_l) \) be an edge realizing this minimum distance, i.e.,  
\[ \bar{d}(u'v', e_{t_l}) = \bar{d}(E(T_l), e_{t_l}). \]  

Without loss of generality, let $d(v',v_{t_l})=d(u'v',e_{t_l})$. Then we claim that $d(v',v_{t_l}) = 2$. Otherwise, if \( d = 1 \), it contradicts that \( T_1 \) is the largest connected component; if \( d \geq 3 \), we can find a path \( u'v'w_1 \cdots  w_kv_{t_l}u_{t_l} (k \geq 2, w_i \notin U)\) connecting \( u'v' \) and \( e_{t_l} \), and thus find a matching \( M + w_1w_2 \) larger than \( M \) since \( w_i \notin U \), which is a contradiction.

Therefore, Let \( P = u'v'w_1v_{t_l}u_{t_l}\) be the unique path in \( T \) connecting \( u'v' \) to \( e_{t_l} \), where \( w_1 \notin U \).  
Replace \( e_{t_l} \) in \( M \) with \( w_1v_{t_l} \). This replacement preserves the matching size \( \beta \) while reducing the number of connected components in \( T[U] \).

\noindent \textbf{Step 4:} Update \( U \) to reflect the modified matching and increment \( l \).

The procedure terminates when \( T[U] \) becomes connected because:  
(i) Each iteration strictly decreases the number of connected components.  
(ii) The total number of components is finite (at most \( \beta \)).  
(iii) The matching remains maximal at every step.  

Upon termination, \( T[U] \) is a connected induced subgraph of the tree \( T \) and thus must itself be a tree containing all edges of $M$. This completes the proof.  
\end{proof}

A vertex is called a \textit{quasi-pendant vertex} if it is a neighbor of a leaf. Recall that \(\mathcal{G}_{n,\beta}\) denotes the set of graphs with order \(n\) and matching number \(\beta\). Our third result demonstrates that there exists a dominating set with special properties in \(T \in \mathcal{T}_{n,\beta}\).

\begin{thm}\label{lem-2.8}
Let \(T \in \mathcal{T}_{n,\beta}\) be a tree with \(n \geq 2\beta + 1\), and let \(Q\) be the set of all quasi-pendant vertices of \(T\). Then there exists a dominating set \(X \subseteq V(T)\) with \(|X| = \beta\) such that:

(i) \(Q \subseteq X\),

(ii) For any two quasi-adjacent vertices \(v_i, v_j \in X\), \(d_T(v_i, v_j) \leq 2\).
\end{thm}

\begin{proof}
By Theorem \ref{lem-2.7}, there exists a maximal matching \(M = \{e_1 = u_1v_1, \ldots, e_\beta = u_\beta v_\beta\}\) of \(T\) such that the induced subgraph \(T[U]\) is a tree, where \(U = \{u_1, v_1, \ldots, u_\beta, v_\beta\}\). Let \(W = V(T) \setminus U\). First, we present the structural properties of \(W\) as follows.

\begin{claim}\label{claim:W-structure}
Every vertex \(w \in W\) is a leaf in \(T\), that is, \(d_W(w) = 0\) and \(d_U(w) = 1\).
\end{claim}

\begin{proof}
Suppose, for contradiction, that some \(w \in W\) has a neighbor in \(W\). Then there exists an edge \(e' = ww' \in E(T[W])\). The set \(M \cup \{e'\}\) would then be a matching of size \(\beta + 1\), contradicting the maximality of \(M\). Hence, \(d_W(w) = 0\).

Next, note that \(d_U(w) \neq 0\) because \(T\) is connected and \(U \neq \emptyset\). Suppose, for contradiction, that \(d_U(w) \geq 2\). Since \(T[U]\) is a tree, \(T[U \cup \{w\}]\) would contain a cycle (as \(w\) is adjacent to two distinct vertices in \(U\)), contradicting the fact that \(T\) is a tree. Thus, \(d_U(w) = 1\).
\end{proof}

\begin{claim}\label{claim:neighbor-empty}
For each edge \(e_i = u_iv_i \in M\), either \(N_W(u_i) = \emptyset\) or \(N_W(v_i) = \emptyset\).
\end{claim}

\begin{proof}
If both \(u_i\) and \(v_i\) had neighbors in \(W\), say \(u_iw\) and \(v_iw'\) with \(w, w' \in W\), then replacing \(e_i\) with \(\{u_iw, v_iw'\}\) in \(M\) would yield a larger matching, contradicting the maximality of \(M\).
\end{proof}

Our goal is to construct a dominating set \(X\) satisfying conditions (i) and (ii). By Claim \ref{claim:neighbor-empty}, we may assume \(N_W(u_i) = \emptyset\) for all \(i \in \{1, \ldots, \beta\}\) and define \(X = \{v_1, \ldots, v_\beta\}\). Since \(M\) covers \(U\), every \(u \in U\) is either in \(X\) or matched to a vertex in \(X\). By Claim \ref{claim:W-structure} and the assumption \(N_W(u_i) = \emptyset\), every \(w \in W\) is a leaf adjacent to some \(v_i \in X\). Thus, \(X\) is a dominating set of \(T\) with \(|X| = \beta\).

\begin{claim}\label{clm-2.3}
\(X\) satisfies condition (i), i.e., \(Q \subseteq X\).
\end{claim}

\begin{proof}
The elements of \(Q\) fall into two categories: \(\{v_i \in X \mid d_W(v_i) \neq \emptyset\}\) and \(\{u_j \notin X \mid u_jv_j \in M \text{ and } v_j \text{ is a leaf of } T\}\). For elements in \[\{u_j \notin X \mid e_j = u_jv_j \in M \text{ and } v_j \text{ is a leaf of } T\},\] swap the labels of the two endpoints of \(e_j = u_jv_j\). Consequently, \(Q \subseteq X\).
\end{proof}

By the choice of \(X\) and Claim \ref{clm-2.3}, we have \(Q \subseteq X \subset U\). Furthermore, we assert that every vertex in \(X\) has at least one quasi-adjacent vertex. To verify this, note that since \(T\) is a tree, for any \(v_i \in X\) and any other \(v_j \in X\) (\(v_j \neq v_i\)), there is a unique path connecting \(v_i\) and \(v_j\) in \(T\). Along this path, select the vertex \(v_k \in X\) (distinct from \(v_i\)) closest to \(v_i\); by definition, \(v_k\) is a quasi-adjacent vertex of \(v_i\), confirming the assertion.

\begin{claim}\label{claim-2.6}
For any two closest quasi-adjacent vertices \(v_i, v_j \in X\), \(d_T(v_i, v_j) \leq 3\).
\end{claim}

\begin{proof}
For contradiction, suppose the distance between \(v_i\) and \(v_j\) (the closest quasi-adjacent vertices in \(X\)) exceeds 3. 
By the definition of quasi-adjacency, there is a path \(P\) connecting \(v_i\) and \(v_j\) in \(T\), 
say \(P = v_iu_pu_qu_r\cdots v_j\) (\(q \neq i, q \neq j\)), where \(u_p, u_q, u_r \in U \setminus X\) 
and the length of \(P\) is at least 4. Note that \(v_qu_q\) is a matching edge in \(T\). 
Consider the path \(P' = v_iu_pu_qv_q\) in \(T\), which implies \(v_q\) is a quasi-adjacent vertex of \(v_i\). 
Moreover, \(v_q\) is closer to \(v_i\) than \(v_j\) is, contradicting the minimality of \(v_j\). 
\end{proof}

Next, we adjust $X$ so that it satisfies condition (ii) while ensuring that condition (i) remains fulfilled. Since \(n \geq 2\beta + 1\), at least one of \(v_1, v_2, \ldots, v_\beta\) is adjacent to vertices in \(W\); without loss of generality, let \(v_1\) be such a vertex. 
We traverse all quasi-adjacent vertices of \(v_1\) in \(X \setminus A\), 
and in each iteration select the closest one, denoted by \(v_{t_1}\). 
By Claim~\ref{claim-2.6}, \(d_T(v_1, v_{t_1}) \le 3\). 
Initialize \(A = \{v_1\}\), and let \(T_0\) be the subgraph induced by \(\{u_1, v_1, u_{t_1}, v_{t_1}\}\) and their neighbors in \(W\).
We then process these quasi-adjacent vertices in the order of Cases~1--4 
(each time denoted by \(v_{t_1}\)) and perform the corresponding operations:

{\flushleft\bf Case 1.} \(d_T(v_1, v_{t_1}) = 1\) and \(d_W(v_{t_1}) = 0\). If \(u_{t_1}\) is not a leaf, swap the labels of \(u_{t_1}\) and \(v_{t_1}\), and then the structure of \(T_0\) becomes (2) in Fig.\ref{fig4}. If \(u_{t_1}\) is a leaf, no operation is performed, and then \(T_0\) has structure (1) in Fig.\ref{fig4}. Notably, label swaps preserve the condition (i).

{\flushleft\bf Case 2.} \(d_T(v_1, v_{t_1}) = 1\) and \(d_W(v_{t_1}) \neq 0\). no operation is performed, and then \(T_0\) has structure (3) in Fig.\ref{fig4}. 

{\flushleft\bf Case 3.} \(d_T(v_1, v_{t_1}) = 2\). No operation is performed, and then \(T_0\) has structure (2), (4), or (5) in Fig.\ref{fig4}.

{\flushleft\bf Case 4.} \(d_T(v_1, v_{t_1}) = 3\). In fact, if there exists a quasi-adjacent vertex of \(v_1\) at distance three, then there also exists a (possibly different) vertex \(v_t\) lying on a path \(v_1u_iu_tv_t\);
thus, we select such a vertex \(v_t\) as \(v_{t_1}\). We assert \(d_W(v_{t_1}) = 0\); otherwise, since \(d_W(v_1) \ge 1\), 
a larger matching than \(M\) would exist, contradicting its maximality. Hence, we swap the labels of \(u_{t_1}\) and \(v_{t_1}\). It is worth noting that the label swaps also preserve condition (i).

After each iteration, update \(A := A \cup \{v_{t_1}\}\) and repeat the process 
until \(X \setminus A\) contains no quasi-adjacent vertices of \(v_1\).

Let \(X_1 = A \setminus \{v_1\}\), which is the set of quasi-adjacent vertices of \(v_1\). 
After the above swap operations, we have \(d_T(v_1, v_{t_1}) \le 2\) for all \(v_{t_1} \in X_1\).
\begin{figure}[t]
   \begin{center}
 \begin{tikzpicture}
[scale=.85,auto=left,every node/.style={circle,scale=.35}]
\foreach \i in {1,2,3,4}{
     \node[draw,circle] (\i) at (\i, 0){};}
\foreach \from/\to in {1/2,2/3,3/4}{
\draw (\from) -- (\to);}
\node[label=below:$\mathbf{u_1}$] at (1, 0){};
\node[label=below:$\mathbf{v_1}$] at (2, 0){};
\node[label=below:$\mathbf{v_{t_1}}$] at (3, 0){};
\node[label=below:$\mathbf{u_{t_1}}$] at (4, 0){};
\node[draw,circle] (a) at (1.5, 1){};
\node[draw,circle] (b) at (2.5, 1){};
\draw (a)--(2)--(b);
\draw[dotted] (a)--(b);
\node[label=below:{\Large(1)}] at (2, -.3){};
\end{tikzpicture}
\begin{tikzpicture}
[scale=.75,auto=left,every node/.style={circle,scale=.35}]
\foreach \i in {1,2,3,4}{
     \node[draw,circle] (\i) at (\i, 0){};}
\foreach \from/\to in {1/2,2/3,3/4}{
\draw (\from) -- (\to);}
\node[label=below:$\mathbf{u_1}$] at (1, 0){};
\node[label=below:$\mathbf{v_1}$] at (2, 0){};
\node[label=below:$\mathbf{u_{t_1}}$] at (3, 0){};
\node[label=below:$\mathbf{v_{t_1}}$] at (4, 0){};
\node[draw,circle] (a) at (1.5, 1){};
\node[draw,circle] (b) at (2.5, 1){};
\draw (a)--(2)--(b);
\draw[dotted] (a)--(b);
\node[label=below:{\Large(2)}] at (2, -.3){};
\end{tikzpicture}
\begin{tikzpicture}
[scale=.85,auto=left,every node/.style={circle,scale=.35}]
\foreach \i in {1,2,3,4}{
     \node[draw,circle] (\i) at (\i, 0){};}
\foreach \from/\to in {1/2,2/3,3/4}{
\draw (\from) -- (\to);}
\node[label=below:$\mathbf{u_1}$] at (1, 0){};
\node[label=below:$\mathbf{v_1}$] at (2, 0){};
\node[label=below:$\mathbf{v_{t_1}}$] at (3, 0){};
\node[label=below:$\mathbf{u_{t_1}}$] at (4, 0){};
\node[draw,circle] (a) at (1.6, 1){};
\node[draw,circle] (b) at (2.3, 1){};
\draw (a)--(2)--(b);
\draw[dotted] (a)--(b);
\node[draw,circle] (a) at (2.6, 1){};
\node[draw,circle] (b) at (3.3, 1){};
\draw (a)--(3)--(b);
\draw[dotted] (a)--(b);
\node[label=below:{\Large(3)}] at (2, -.3){};
\end{tikzpicture}
\begin{tikzpicture}
[scale=.75,auto=left,every node/.style={circle,scale=.35}]
\foreach \i in {1,2,3,4}{
     \node[draw,circle] (\i) at (\i, 0){};}
\foreach \from/\to in {1/2,2/3,3/4}{
\draw (\from) -- (\to);}
\node[label=below:$\mathbf{v_1}$] at (1, 0){};
\node[label=below:$\mathbf{u_1}$] at (2, 0){};
\node[label=below:$\mathbf{v_{t_1}}$] at (3, 0){};
\node[label=below:$\mathbf{u_{t_1}}$] at (4, 0){};
\node[draw,circle] (a) at (.5, 1){};
\node[draw,circle] (b) at (1.5, 1){};
\draw (a)--(1)--(b);
\draw[dotted] (a)--(b);
\node[label=below:{\Large(4)}] at (2, -.3){};
\end{tikzpicture}
\begin{tikzpicture}
[scale=.75,auto=left,every node/.style={circle,scale=.35}]
\foreach \i in {1,2,3,4}{
     \node[draw,circle] (\i) at (\i, 0){};}
\foreach \from/\to in {1/2,2/3,3/4}{
\draw (\from) -- (\to);}
\node[label=below:$\mathbf{v_1}$] at (1, 0){};
\node[label=below:$\mathbf{u_1}$] at (2, 0){};
\node[label=below:$\mathbf{v_{t_1}}$] at (3, 0){};
\node[label=below:$\mathbf{u_{t_1}}$] at (4, 0){};
\node[draw,circle] (a) at (.5, 1){};
\node[draw,circle] (b) at (1.5, 1){};
\draw (a)--(1)--(b);
\draw[dotted] (a)--(b);
\node[draw,circle] (a) at (2.5, 1){};
\node[draw,circle] (b) at (3.5, 1){};
\draw (a)--(3)--(b);
\draw[dotted] (a)--(b);
\node[label=below:{\Large(5)}] at (2, -.3){};
\end{tikzpicture}
\end{center}
\vspace*{-.8cm}
\caption{Five possible structures.}\label{fig4}
\end{figure}
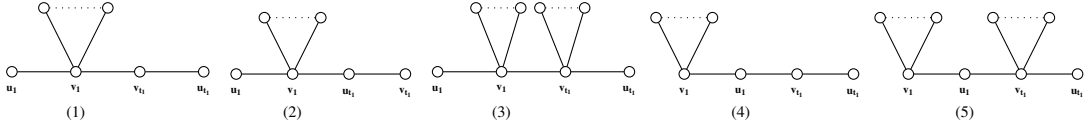

We first traverse the quasi-adjacent vertices of \(v_1\) to obtain \(X_1\). 
Then, for each vertex in \(X_1\), we successively find all of its quasi-adjacent vertices that are closest to it, 
whose union forms \(X_2\).
Similarly, for each vertex in \(X_2\), we obtain all its quasi-adjacent vertices, 
whose union forms \(X_3\), and so on (see Fig.\ref{2-8-4}). 
This process continues iteratively until all vertices in \(X\) have been traversed. At each iteration, the corresponding operations are performed according to the four cases (Case 1–Case 4). As a result, we obtain \(d_T(v_{t_i}, v_{t_{i+1}}) \leq 2(i=1,2,\dots,k)\) for \(v_{t_i} \in X_i\) and \(v_{t_{i+1}} \in X_{i+1}\) with \(v_{t_i} \dot{\sim} v_{t_{i+1}}\). 
\begin{figure}[h]
\centering
\begin{tikzpicture}[scale=0.025]

\tikzset{
    vertex/.style={draw, circle, fill=black, inner sep=2pt},
    vblue/.style={draw, circle, fill=blue, inner sep=2.2pt}
}


\draw[thick, rounded corners=10pt]
  (38.5,32.5) rectangle (69.5,153.5);

\draw[thick, rounded corners=10pt]
  (93.4,15) rectangle (126.6,163);

\draw[thick, rounded corners=10pt]
  (159.4,15.5) rectangle (192.6,163.5);
  
\draw[thick, rounded corners=10pt]
(248,17) rectangle (283,165);

\draw[thick, rounded corners=10pt]
(332,17) rectangle (367.6,165);

\draw[thick, rounded corners=10pt]
(419.4,17) rectangle (452.6,165);


\node[vertex] (v1) at (2,100.5) {};
\node[vertex] at (55,50.5) {};
\node[vblue]  (vt1) at (55,97) {};
\node[vertex] at (55,132) {};

\draw (56,133) -- (112,153.5);
\draw (58,132) -- (110,124.5);
\draw (2.6,100) -- (53.6,133);
\draw (3.12,99.5) -- (52.12,52.5);

\draw (v1) -- (vt1);

\node at (56,2.5) {$X_1$};
\node at (55,104) {$v_{t_1}$};
\node at (2.6,90) {$v_1$};
\node at (55,117) {$\vdots$};
\node at (55,74) {$\vdots$};

\node[vertex] at (110,152.5) {};
\node[vertex] at (110,124.5) {};
\node[vblue]  (vt2) at (110,109) {};
\node[vertex] (x2_a) at (110,80.5) {};
\node[vertex] at (110,54.5) {};
\node[vertex] at (110,28) {};

\draw (54.6,51) -- (111.12,54.5);
\draw (54.12,49.5) -- (110.12,28.5);

\draw (111.12,124) -- (173.12,151);
\draw (110.62,123.5) -- (175.62,108);
\draw (110.12,109) -- (176,93);

\node at (109,2.5) {$X_2$};
\node at (110,116) {$v_{t_2}$};
\node at (110,138.5) {$\vdots$};
\node at (110,95) {$\vdots$};
\node at (110,41) {$\vdots$};

\node[vertex] at (176,150.5) {};
\node[vertex] at (176,108.5) {};
\node[vblue]  (vt3) at (176,93) {};
\node[vertex]  (below_vt3) at (176,79.5) {};
\node[vertex] at (176,37.5) {};

\draw (111,56) -- (176.12,79.5);
\draw (111,56)--(176,37.5);

\node at (174.6,2.5) {$X_3$};
\node at (176,100) {$v_{t_3}$};
\node at (176,133) {$\vdots$};
\node at (176,57) {$\vdots$};


\node[vblue] at (266,115) {};
\node[vblue] at (266,71) {};
\node at (268,124) {$v_{i-1}$};
\node at (268,54) {$v_{i-1}'$};
\node at (266,148) {$\vdots$};
\node at (266,98) {$\vdots$};
\node at (266,40) {$\vdots$};

\node at (267.6,2.5) {$X_{i-1}$};
\node[vblue] at (350,115) {};
\node[vblue] at (350,71) {};
\node at (352,124) {$v_{t_i}$};
\node at (352,54) {$v_{t_i}'$};
\node at (350,148) {$\vdots$};
\node at (350,40) {$\vdots$};

\node at (351.6,2.5) {$X_{i}$};

\draw (266,115) -- (350,115);
\draw (266,71) -- (350,71);
\draw (350,115) -- (350,71);

\node[vertex] at (436,152) {};
\node[vertex] at (436,93) {};
\node[vertex] at (436,31.8) {};

\node at (437.6,2.5) {$X_k$};
\node at (436,124) {$\vdots$};
\node at (436,62) {$\vdots$};

\node at (220,98.5) {$\dots\dots$};
\node at (395,98.5) {$\dots\dots$};

\draw (56,97) -- (108.12,108.5);
\draw (54.12,96.5) -- (109.62,80);

\end{tikzpicture}
\caption{A partition of $X$.}
\label{2-8-4}
\end{figure}
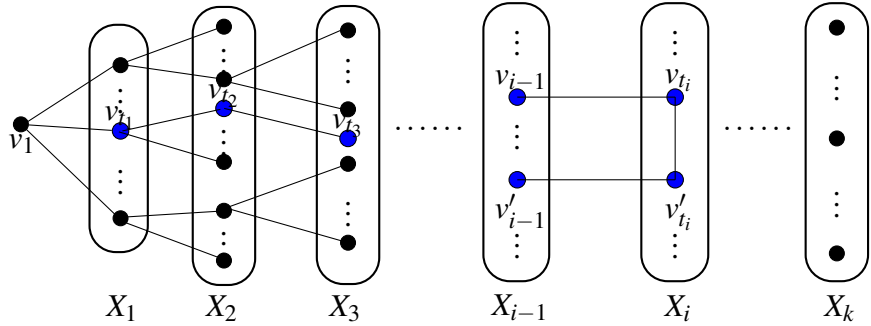

Next, we only need to consider the distances between two quasi-adjacent vertices in \(X_i\). 
In fact,  for every $X_i$, if \(v_{t_i}, v_{t_i}' \in X_i\) and \(v_{t_i} \dot{\sim} v_{t_i}'\), 
 then there must exist quasi-adjacent vertices, say  $v_{i-1}$,  $v_{i-1}'$,  in $X_{i-1}$.
 If $v_{i-1} \neq v_{i-1}'$, then $T$ will create a cycle, a contradiction.
 If $v_{i-1} = v_{i-1}'$, then 
 $P_{v_{i-1}v_{t_i}} = v_{i-1}u_jv_{t_i}$, $P_{v_{i-1}v_{t_i}'} = v_{i-1}u_jv_{t_i}'$, where $u_j \in X_j$,
 otherwise $T$ will create a cycle. Therefore, we have \(d(v_{t_i}, v_{t_i}') = 2\).

Consequently, \(X\) satisfies condition (ii), 
that is, for any two quasi-adjacent vertices \(v_i, v_j \in X\), \(d_T(v_i, v_j) \le 2\).
This completes the proof.
\end{proof}

\begin{thm}\label{cliqueblock}
For any cycle in the quasi-adjacent graph with respect to a control set of a tree, where the control set is chosen according to Theorem~2.3, every pair of vertices within that cycle is quasi-adjacent.

\end{thm}

\begin{figure}[h]
	\centering
	\begin{tikzpicture}[scale=0.65,
		vertex/.style={circle, fill=black, minimum size=5pt, inner sep=0pt},
		solid edge/.style={thick, smooth},
		dashed edge/.style={thick, smooth, dashed},
		label/.style={font=\footnotesize, outer sep=1pt},
		rotate=90 
		]
		
		\node[vertex] (v1) at (120:3.5cm) {};
		\node[vertex] (vk) at (120:2cm) {};
		\node[vertex] (vi) at (60:3.5cm) {};
		\node[vertex] (vi1) at (0:3.5cm) {};
		\node[vertex] (vi2) at (-60:3.5cm) {};
		\node[vertex] (vj) at (-120:2cm) {};
		
		\draw[dashed edge] (v1) to[out=-30, in=150] (vi);
		\draw[solid edge] (vi) to[out=-90, in=90] (vi1);
		\draw[solid edge] (vi1) to[out=-150, in=30] (vi2);
		\draw[dashed edge] (vi2) to[out=180, in=-60] (vj);
		\draw[solid edge] (v1) -- (vk);
		\draw[dashed edge] (vk) to[out=-60, in=120] (vj);
		
		\node[label, at=(v1.180)] [xshift=-3pt, yshift=3pt] {$v_1$};
		\node[label, at=(vk.180)] [xshift=-3pt] {$v_k$};
		\node[label, at=(vi.90)] [yshift=3pt] {$v_i$};
		\node[label, at=(vi1.0)] [xshift=3pt] {$v_{i+1}$};
		\node[label, at=(vi2.315)] [xshift=3pt, yshift=-3pt] {$v_{i+2}$};
		\node[label, at=(vj.225)] [xshift=-3pt, yshift=-3pt] {$v_j$};
		
	\end{tikzpicture}
	\caption{\small{A cycle $C$ in the quasi-adjacent graph $T_X$, where a dashed edge between two vertices represents the existence of a path connecting them in $T_X$.}}\label{fig-1}
\end{figure}
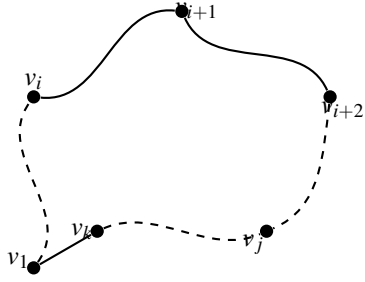

\begin{proof}
		Suppose there is a cycle \(C = v_1v_2\cdots v_kv_1\) of length $k$ in the quasi-adjacent graph $T_X$ of the tree $T$, where $X$ is a control set 
		 chosen in accordance with Theorem \ref{lem-2.8} of $T$. 
	We aim to show that any two vertices \(v_i\) and \(v_j\) on $C$ (with \(i < j\)) are quasi-adjacent.
	
	If \(j = i+1\) or \((i=1 \text{ and } j=k)\), the conclusion follows immediately from the definition of a cycle 
	(since consecutive vertices in a cycle are connected by an edge, hence quasi-adjacent by construction of the quasi-adjacent graph).
	For the remaining case where \(v_i\) and \(v_j\) are not consecutive in $C$, suppose, by contradiction, 
	that \(v_i\) and \(v_j\) are not quasi-adjacent. 
	Their positions in this scenario are illustrated in Figure \ref{fig-1}.
	
	We now establish the distances in $T$ between consecutive vertices in the cycle of $T_X$.
	\begin{claim}\label{clm-a}
		$d_T(v_i, v_{i+1}) = 2$ and $d_T(v_{i+1}, v_{i+2}) = 2$.
	\end{claim}
	
	\begin{proof}
	        By Theorem \ref{lem-2.8} ,  we have $d_T(v_i, v_{i+1}) \leq 2$  and $d_T(v_{i+1}, v_{i+2}) \leq 2$.
		Suppose to the contrary  that $d_T(v_i, v_{i+1}) = 1$ or $d_T(v_{i+1}, v_{i+2}) = 1$. Then there exists a path in $T$, say $P'$, between $v_i$ and $v_{i+2}$ that contains the vertex $v_{i+1}$.
		Clearly, 
		$$P_{v_i, v_{i+2}} \in \left\{ v_i v_{i+1} v_{i+2}, \, v_i u v_{i+1} v_{i+2}, \, v_i v_{i+1} w v_{i+2} \right\} \quad (\text{where } u, w \in V(T) \setminus X).$$
		
		Note that in the quasi-adjacent graph $T_X$, the vertex sequence on $C$ satisfies the following relation:
		$v_i \dot{\sim} v_{i-1} \dot{\sim} \cdots \dot{\sim} v_1 \dot{\sim} v_k \dot{\sim} v_{k-1} \dot{\sim} \cdots \dot{\sim} v_j \dot{\sim} \cdots \dot{\sim} v_{i+2}.$
		Notice that $v_{i+1} \in C$ in $T_X$. By the definition of the quasi-adjacent relation, there exists a sequence of paths in $T$:
		$$P_{v_i, v_{i-1}}, \, P_{v_{i-1}, v_{i-2}}, \, \cdots, \, P_{v_1, v_k}, \, P_{v_k, v_{k-1}}, \, \cdots, \, P_{v_{i+3}, v_{i+2}}.$$
		Moreover, none of these paths pass through the vertex $v_{i+1}$. Hence, between $v_i$ and $v_{i+2}$, there exists another path in $T$, denoted $P''$, that does not contain the vertex $v_{i+1}$.
		Clearly, $P' \neq P''$. It follows that $v_i P' v_{i+2} P'' v_i$ forms a cycle in tree $T$, a contradiction.
	\end{proof}

	\begin{claim}\label{clm-b}
		$v_i$ and $v_{i+2}$ are quasi-adjacent in $T_X$.
	\end{claim}
	
	\begin{proof}
		By Claim \ref{clm-a}, there exist two paths in $T$:
		\[ P_1 = v_i u_1 v_{i+1} \quad (u_1 \notin X) \quad \text{and} \quad P_2 = v_{i+1} u_2 v_{i+2} \quad (u_2 \notin X). \]
		
		We first show that $u_1 = u_2$. Suppose, by contradiction, that $u_1 \neq u_2$. Then $T$ contains the path
		$ P_3 = v_i u_1 v_{i+1} u_2 v_{i+2}$, 
		which passes through $v_{i+1}$. From the proof of Claim \ref{clm-a}, there exists another path $P''$ in $T$ connecting $v_i$ and $v_{i+2}$ that does not contain $v_{i+1}$. This implies $P_3 \neq P''$, so the cycle
		$v_i P_3 v_{i+2} P'' v_i$ exists in $T$, contradicting the fact that $T$ is a tree. Thus, $u_1 = u_2$.
		
		Let $u = u_1 = u_2$. Then $T$ contains the path
		$ P_{v_iv_{i+2}} = v_i u v_{i+2} $
		with $u \notin X$. By the definition of quasi-adjacency in $T_X$, $v_i$ and $v_{i+2}$ are quasi-adjacent.
	\end{proof}
	
	By Claim \ref{clm-b}, we may form a cycle \( C' \) in \( T_X \) of length \( k-1 \) by substituting the path \( v_i v_{i+1} v_{i+2} \) in \( C \) with the quasi-adjacent edge \( v_i \dot{\sim} v_{i+2} \), yielding \( C' = v_1 v_2 \cdots v_i v_{i+2} \cdots v_k v_1 \).
	
	We proceed by induction on the distance between vertices along \( C \). For the base case, Claim \ref{clm-b} establishes quasi-adjacency for vertices with two intermediate vertices (i.e., distance 2 in \( C \)). For the inductive step, consider vertices \( v_i \) and \( v_{i+3} \): since \( v_i \dot{\sim} v_{i+2} \) (by Claim \ref{clm-b}) and \( v_{i+2} \dot{\sim} v_{i+3} \) (as consecutive vertices in \( C \)), an argument analogous to the proof of Claim \ref{clm-b} (applying the cycle \( C' \) in place of \( C \)) shows \( v_i \dot{\sim} v_{i+3} \).
	
Iterating this reasoning, we conclude \(v_i \dot{\sim} v_m\) in \(T_X\) for all m satisfying \(i < m \leq k\). Hence, all pairs \(v_i, v_j \in C\) with \(i < j\) are quasi-adjacent, completing the proof.
\end{proof}

A clique block in a graph $G$ is a maximal clique (a complete subgraph not contained in any larger complete subgraph) whose vertex set is a cut set of $G$, or a trivial clique (single vertex) that forms a cut vertex of $G$. In particular, an edge (2-clique) that constitutes a bridge in $G$ is called a trivial clique block. For a tree $T$, Theorem \ref{cliqueblock} immediately implies the following result:
\begin{cor}\label{cor-2.1}
Every quasi-adjacent graph with respect to a control set chosen in accordance with Theorem \ref{lem-2.8} of $T$ decomposes into a collection of clique blocks (including 2-cliques).
\end{cor}


\section{The characteristic of minimizers}\label{sec-3}

To investigate the structure of graphs with minimal spectral radius within the family \(\mathcal{G}_{n,\beta}\), we first establish the following result.

\begin{thm}\label{thm-2.6}
Among all graphs in $\mathcal{G}_{n,\beta}$, the graph with minimal spectral radius must be a tree.
\end{thm}

\begin{proof}
Let $G^*$ be a graph in $\mathcal{G}_{n,\beta}$ with minimal spectral radius $\rho(G^*)$. Suppose, for contradiction, that $G^*$ is not a tree. 
By Theorem \ref{spanning_tree}, there exists a spanning tree $T$ of $G^*$ that belongs to $\mathcal{G}_{n,\beta}$ and has the same matching number $\beta$. Since $T$ is a proper subgraph of $G^*$ (as $G^*$ contains cycles while $T$ is acyclic), it follows from Lemma \ref{propersubgraph} that $\rho(T) < \rho(G^*)$. 
This contradicts the minimality of $\rho(G^*)$ in $\mathcal{G}_{n,\beta}$. Therefore, $G^*$ must be a tree.
\end{proof}

For any tree \(T \in \mathcal{T}_{n,\beta}\) with \(n \geq 2\beta + 1\), let \(Q\) denote the set of all quasi-pendant vertices of \(T\). By Theorem \ref{lem-2.8}, there exists a dominating set \(X \subseteq V(T)\) of size \(|X| = \beta\) such that:

(i) \(Q \subseteq X\),

(ii) For any two quasi-adjacent vertices \(v_i, v_j \in X\), \(d_T(v_i, v_j) \leq 2\).

Clearly, every tree \(T \in \mathcal{T}_{n,\beta}\) with \(n \geq 2\beta + 1\) has a quasi-adjacent graph of order $\beta$. 
Moreover, by the definition of quasi-adjacency and Corollary \ref{cor-2.1}, the quasi-adjacency graph is composed by clique blocks. 

\begin{figure}[h]
  \centering
\begin{tikzpicture}
[scale=1,auto=left,every node/.style={draw,circle,inner sep=1.1pt}]
    \node[fill,inner sep=1.8pt] (a1) at (-1.5,0) {};
    \node[fill,inner sep=1.8pt] (a2) at (0,0) {};
    \draw (a1) -- (a2);

    \draw[->, thick] (1,0) -- (2.5,0);

    \node[draw,circle] (b1) at (4.5,0) {};
    \node (b2) at (6,0) {};

    \node (t11) at (4.1,1) {};
    \node (t13) at (4.9,1) {};

    \draw (b1) -- (t11);
    \draw (b1) -- (t13);
    \draw[dotted] (t11) -- (t13);

    \node (t21) at (5.6,1) {};;
  \node (t23) at (6.4,1) {};
    \draw (b2) -- (t21);
  \draw (b2) -- (t23);
    \draw[dotted] (t21) -- (t23);

    \draw (b1) -- (b2);

    \node[draw=none, fill=none] at (7,0) {or};

    \node (c1) at (8,0) {};
    \node (c2) at (9,0) {};
    \node (c3) at (10,0) {};

    \draw (c1) -- (c2) -- (c3);

    \node (u1) at (7.5,1) {};
    \node (u3) at (8.5,1) {};

    \draw (c1) -- (u1);
    \draw (c1) -- (u3);
    \draw[dotted] (u1) -- (u3);

    \node (w1) at (9.5,1) {};
    \node (w3) at (10.5,1) {};

    \draw (c3) -- (w1);
    \draw (c3) -- (w3);
    \draw[dotted] (w1) -- (w3); 
    
      \node[draw=none, fill=none] at (-1,-.5) {quasi-adjacent edge};
\end{tikzpicture}

\vspace*{-1cm}
\begin{tikzpicture}[scale=1, every node/.style={circle, fill, inner sep=1.8pt}]
    \node (a1) at (0:1)   {};
    \node (a2) at (60:1)  {};
    \node (a3) at (120:1) {};
    \node (a4) at (180:1) {};
    \node (a5) at (240:1) {};
    \node (a6) at (300:1) {};

    \draw (a1) -- (a2);
    \draw (a1) -- (a3);
    \draw (a1) -- (a4);
    \draw (a1) -- (a5);
    \draw (a1) -- (a6);

    \draw (a2) -- (a3);
    \draw (a2) -- (a4);
    \draw (a2) -- (a5);
    \draw (a2) -- (a6);

    \draw (a3) -- (a4);
    \draw (a3) -- (a5);
    \draw (a3) -- (a6);

    \draw (a4) -- (a5);
    \draw (a4) -- (a6);

    \draw[dotted] (a5) -- (a6);

    \node[draw=none, fill=none] at (0,-1.3) {{$k$-clique ($k\ge 3$)}};
    \draw[->, thick] (2,0) -- (4,0);
     \node[draw=none, fill=none] at (5.5,0) {   ~~~~~~~~~~~~ };
\end{tikzpicture}
\begin{tikzpicture}[scale=.65,auto=left,every node/.style={draw,circle,scale=.35}]
    >=stealth]

\node (C) at (0,0) {};      
\node (L) at (-2.8,0) {};     
\node (R) at ( 2.8,0) {};     
\node (D) at (-2,-2) {};     
\node (UD) at (2,-2) {};     
\node (UL) at (-2,2) {};    
\node (UR) at ( 2,2) {};    

\draw (C)--(L);
\draw (C)--(R);
\draw (C)--(D);
\draw (C)--(UD);
\draw (C)--(UL);
\draw (C)--(UR);

\node (UL1) at (-3,3.5) {};
\node (UL2) at (-1,3.5) {};
\draw (UL)--(UL1);
\draw (UL)--(UL2);
\draw[dotted] (UL1) -- (UL2);

\node (UR1) at (1,3.5) {};
\node (UR2) at (3,3.5) {};
\draw (UR)--(UR1);
\draw (UR)--(UR2);
\draw[dotted] (UR1) --(UR2);

\node (DL1) at (-3,-3.5) {};
\node (DL2) at ( -1,-3.5) {};
\draw (D)--(DL1);
\draw (D)--(DL2);
\draw[dotted] (DL1)--(DL2);

\node (DR1) at (1,-3.5) {};
\node (DR2) at (3,-3.5) {};
\draw (UD)--(DR1);
\draw (UD)--(DR2);
\draw[dotted] (DR1)--(DR2);
\draw[dotted] (D)--(UD);

\node (L1) at (-4.5,1) {};
\node (L2) at (-4.5,-1) {};

\draw (L)--(L1);
\draw (L)--(L2);

\draw[dotted] (L1) -- (L2);

\node (R1) at (4.5,1) {};
\node (R2) at (4.5,-1) {};
\draw (R)--(R1);
\draw (R)--(R2);
\draw[dotted] (R1) -- (R2);
\end{tikzpicture}
\caption{Map the quasi-adjacent graph to the original graph $T$.}\label{fig-main-tree-beta}
\end{figure}
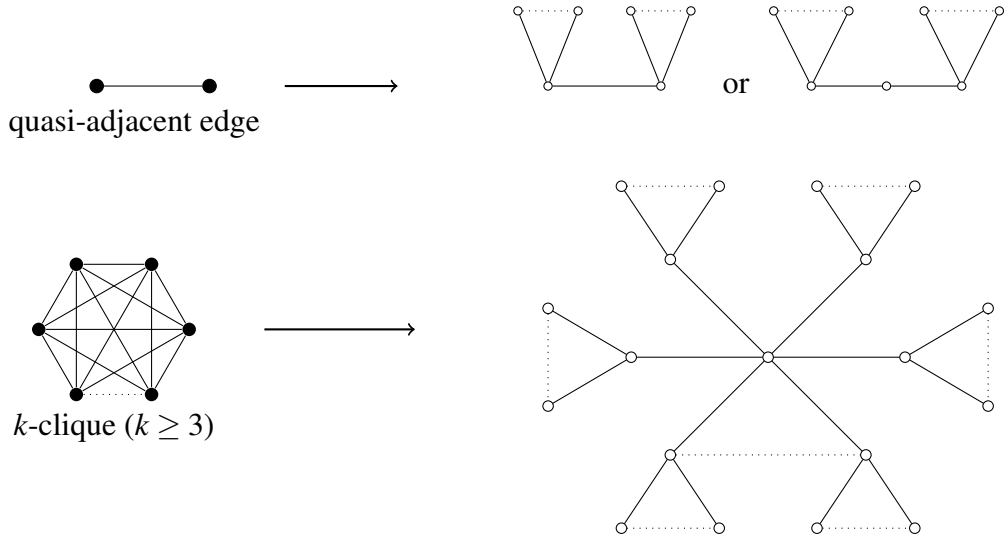

\subsection{The minimizers for $\beta <4$}

To characterize the minimizers, we are going to use heavily the following well known results.
The first one tells that subdividing an internal edge, e.g. an edge that is not in a pendant path, decreases the spectral radius. 
\begin{lem}[\cite{HoffmanSmith1975}]\label{subdivision}
		Let $ W_{n-2} $ be the graph obtained by connecting two vertices to the two quasi-pendant vertices of a path with $n - 2$ vertices, respectively. Assume that 
		$G \not\cong W_{n-2}$ and $uv$ is an edge on an internal path of $G$. Let $G_{uv}$ be the graph obtained from G by the subdivision of the edge $uv$. Then $\rho(G)>\rho(G_{uv})$.
\end{lem}

\begin{lem}[\cite{LiFeng1979}]\label{propersubgraph}
	Let $G_1$ and $G_2$ be two graphs. If $G_2$ is a proper subgraph of $G_1$, then $\rho(G_2)<\rho(G_1)$.
\end{lem}

\begin{lem}[\cite{LinGuo2006}]\label{average-min}
	Let $u$ and $v$ be two vertices in a connected graph $G$ such that $G - u \cong G - v$. 
	Denote by $G_{k,l}^2$ the graph obtained from $G$ by adding $k$ pendant edges at vertex $u$ and $l$ pendant edges at vertex $v$. 
	If $k \ge l \ge 1$, then 
	$\lambda_1\!\left(G_{k+1,l-1}^2\right) > \lambda_1\!\left(G_{k,l}^2\right)$.
\end{lem}

\begin{proof}[\bf{Proof of Theorem \ref{thm-1.1}}]
	By Theorem~\ref{thm-2.6}, $T_2^*\in \mathcal{G}_{n,2}$ is a tree. Then we can get a dominating set $X$ of $T_2^*$ with $|X|=2$ which satisfies conditions (i) and (ii) by Theorem \ref{lem-2.8}.
	Moreover,  the quasi-adjacency graph of $X$ are  block graphs by Corollary \ref{cor-2.1}.
	In fact, the quasi-adjacency graph of $X$ is $P_2$. There are only two types of graphs, as illustrated in Fig.~\ref{fig-main-tree-beta} (top). 
	Let's us denominate by $T_1(a,b)$ the tree on the left, having two central nodes and by $T_2(a,b)$ the tree on the right. 
	Our goal is to show that, for a given $n\geq 4$, $T_2(a,b)$ is the one minimizing the spectral radius. Suppose that the conclusion does not hold. 
	Then \(T_2^*\) is isomorphic to \(T_1(a,b)\), where \(a,b \ge 1\). 
Without loss of generality, we may assume that both \(a,b \ge 2\); otherwise, \(T_1(a,b)\) can be viewed as \(T_2(a,b-1)\) or \(T_2(a-1,b)\). 
The graph \(T_2(a-1,b)\) can be obtained from \(T_1(a,b)\) by first subdividing the unique internal edge and then deleting one pendant vertex. 
By Lemmas~\ref{subdivision} and~\ref{propersubgraph}, we have 
\[
\rho(T_2(a-1,b)) < \rho(T_1(a,b)),
\]
which contradicts the minimality of the spectral radius of \(T_2^*\).
Therefore, \( T_2^* \) is isomorphic to \( T_2(a,b) \). Now, we invoke  Lemma~\ref{average-min} to show that among all $T_2(a,b)$ the minimizer it the balanced.
\(
T_2^* \cong T_2\!\left(\lfloor \tfrac{n-3}{2} \rfloor,\, \lceil \tfrac{n-3}{2} \rceil\right).
\)
\end{proof}

\begin{figure}[h]
\centering
\begin{tikzpicture}
[scale=.85,auto=left,every node/.style={draw,circle,scale=.45}]
\node (A1) at (0,0) [circle,fill,inner sep=2pt] {};
\node (A2) at (1.2,0) [circle,fill,inner sep=2pt] {};
\node (A3) at (2.4,0) [circle,fill,inner sep=2pt] {};
\draw (A1)--(A2)--(A3);

\draw[->, thick] (3.0,-1.2) -- (4.2,-1.2);

\node (B1) at (0,-3) [circle,fill,inner sep=2pt] {};
\node (B2) at (1.2,-3) [circle,fill,inner sep=2pt] {};
\node (B3) at (0.6,-2) [circle,fill,inner sep=2pt] {};
\draw (B1)--(B2)--(B3)--(B1);

\begin{scope}[shift={(5,0.3)}]
\node (v1) at (0,0) [circle,inner sep=2pt,label=below:{$v_1$}] {};
\node (v2) at (1.5,0) [circle,inner sep=2pt,label=below:{$v_2$}] {};
\node (v3) at (3,0) [circle,label=below:{$v_3$},inner sep=2pt] {};
\draw (v1)--(v2)--(v3);

\node (a1) [inner sep=2pt] at (-0.5,0.8)  {};
\node (a2) [inner sep=2pt] at (0.5,0.8)  {};
\draw (v1)--(a1);
\draw (v1)--(a2);
\draw [decorate,decoration={brace,mirror,amplitude=4pt},xshift=0.4pt,yshift=-0.4pt]
(.6,.9) -- (-.6,.9) node [black,draw=none,midway,yshift=.8cm]{{\small $a$}};

\node (b1) at (1,0.8) [inner sep=2pt] {};
\node (b2) at (2,0.8) [inner sep=2pt] {};
\draw (v2)--(b1);
\draw (v2)--(b2);
\draw [decorate,decoration={brace,mirror,amplitude=4pt},xshift=0.4pt,yshift=-0.4pt]
(2.1,.9) --(.9,.9) node [black,draw=none,midway,yshift=.8cm]{{\small $b$}};

\node (c1) at (2.5,0.8) [inner sep=2pt] {};
\node (c2) at (3.5,.8) [inner sep=2pt] {};
\draw (v3)--(c1);
\draw (v3)--(c2);
\draw [decorate,decoration={brace,mirror,amplitude=4pt},xshift=0.4pt,yshift=-0.4pt]
(3.6,.9) -- (2.4,.9) node [black,draw=none,midway,yshift=.8cm]{{\small $c$}};

\node[draw=none] at (1.3,-0.8) {$T_1(a,b,c)$};
\end{scope}

\begin{scope}[shift={(10,0.3)}]
\node (v1) at (0,0) [circle,inner sep=2pt,label=below:{$v_1$}] {};
\node (v2) at (1.5,0) [circle,inner sep=2pt,label=below:{$v_2$}] {};
\node (v3) at (3,0) [circle,inner sep=2pt,label=below:{$v_3$}] {};
\node (v) at (.75,0) [circle,inner sep=2pt] {};
\draw (v1)--(v)--(v2)--(v3);

\node (a1) at (-0.5,0.8) [inner sep=2pt] {};
\node (a2) at (0.5,0.8) [inner sep=2pt] {};
\draw (v1)--(a1);
\draw (v1)--(a2);
\draw [decorate,decoration={brace,mirror,amplitude=4pt},xshift=0.4pt,yshift=-0.4pt]
(.6,.9) -- (-.6,.9) node [black,draw=none,midway,yshift=.8cm]{{\small $a$}};

\node (b1) at (1,0.8) [inner sep=2pt] {};
\node (b2) at (2,0.8) [inner sep=2pt] {};
\draw (v2)--(b1);
\draw (v2)--(b2);
\draw [decorate,decoration={brace,mirror,amplitude=4pt},xshift=0.4pt,yshift=-0.4pt]
(2.1,.9) -- (.9,.9) node [black,draw=none,midway,yshift=.8cm]{{\small $b$}};

\node (c1) at (2.5,0.8) [inner sep=2pt] {};
\node (c2) at (3.5,.8) [inner sep=2pt] {};
\draw (v3)--(c1);
\draw (v3)--(c2);
\draw [decorate,decoration={brace,mirror,amplitude=4pt},xshift=0.4pt,yshift=-0.4pt]
(3.6,.9) -- (2.4,.9) node [black,draw=none,midway,yshift=.8cm]{{\small $c$}};

\node[draw=none] at (1.3,-0.8) {$T_2(a,b,c)$};
\end{scope}

\begin{scope}[shift={(5.3,-3.5)}]
\node (v1) at (0,0) [circle,inner sep=2pt,label=below:{$v_1$}] {};
\node (v2) at (1.5,0) [circle,inner sep=2pt,label=below:{$v_2$}] {};
\node (v3) at (3,0) [circle,inner sep=2pt,label=below:{$v_3$}] {};
\node (v4) at (.75,0) [circle,inner sep=2pt] {};
\node (v5) at (2.25,0) [circle,inner sep=2pt] {};
\draw (v1)--(v4)--(v2)--(v5)--(v3);

\node (a1) at (-0.5,0.8) [inner sep=2pt] {};
\node (a2) at (0.5,0.8) [inner sep=2pt] {};
\draw (v1)--(a1);
\draw (v1)--(a2);
\draw [decorate,decoration={brace,mirror,amplitude=4pt},xshift=0.4pt,yshift=-0.4pt]
(.6,.9) -- (-.6,.9) node [black,draw=none,midway,yshift=.8cm]{{\small $a$}};

\node (b1) at (1,0.8) [inner sep=2pt] {};
\node (b2) at (2,0.8) [inner sep=2pt] {};
\draw (v2)--(b1);
\draw (v2)--(b2);
\draw [decorate,decoration={brace,mirror,amplitude=4pt},xshift=0.4pt,yshift=-0.4pt]
(2.1,.9) -- (.9,.9) node [black,draw=none,midway,yshift=.8cm]{{\small $b$}};

\node (c1) at (2.5,0.8) [inner sep=2pt] {};
\node (c2) at (3.5,.8) [inner sep=2pt] {};
\draw (v3)--(c1);
\draw (v3)--(c2);
\draw [decorate,decoration={brace,mirror,amplitude=4pt},xshift=0.4pt,yshift=-0.4pt]
(3.6,.9) -- (2.4,.9) node [black,draw=none,midway,yshift=.8cm]{{\small $c$}};

\node[draw=none] at (1.3,-0.8) {$T_3(a,b,c)$};
\end{scope}

\begin{scope}[shift={(10,-3.5)}]
\node (v1) at (0,0) [circle,inner sep=2pt,label=below:{$v_1$}] {};
\node (v2) at (1.5,0) [circle,inner sep=2pt,label=below:{$v_2$}] {};
\node (v3) at (3,0) [circle,inner sep=2pt,label=below:{$v_3$}] {};

\draw (v1)--(v2)--(v3);

\node (a1) at (-0.5,0.8) [inner sep=2pt] {};
\node (a2) at (0.5,0.8) [inner sep=2pt] {};
\draw (v1)--(a1);
\draw (v1)--(a2);
\draw [decorate,decoration={brace,mirror,amplitude=4pt},xshift=0.4pt,yshift=-0.4pt]
(.6,.9) -- (-.6,.9) node [black,draw=none,midway,yshift=.8cm]{{\small $a$}};

\node (b1) at (1.5,0.8) [inner sep=2pt] {};
\draw (v2)--(b1);
\node (b2) at (1,1.6) [inner sep=2pt] {};
\node (b3) at (2,1.6) [inner sep=2pt] {};
\draw (b1)--(b2);\draw (b1)--(b3);
\draw [decorate,decoration={brace,mirror,amplitude=4pt},xshift=0.4pt,yshift=-0.4pt]
(2.1,1.7) -- (.9,1.7) node [black,draw=none,midway,yshift=.8cm]{{\small $c$}};

\node (c1) at (2.5,0.8) [inner sep=2pt] {};
\node (c2) at (3.5,.8) [inner sep=2pt] {};
\draw (v3)--(c1);
\draw (v3)--(c2);
\draw [decorate,decoration={brace,mirror,amplitude=4pt},xshift=0.4pt,yshift=-0.4pt]
(3.6,.9) -- (2.4,.9) node [black,draw=none,midway,yshift=.8cm]{{\small $b$}};

\node[draw=none] at (1.3,-0.8) {$T_4(a,b,c)$};
\end{scope}
\end{tikzpicture}
\vspace*{-.8cm}
\caption{$\mathcal{T}_{n,3}$.}\label{fig-7}
\end{figure}
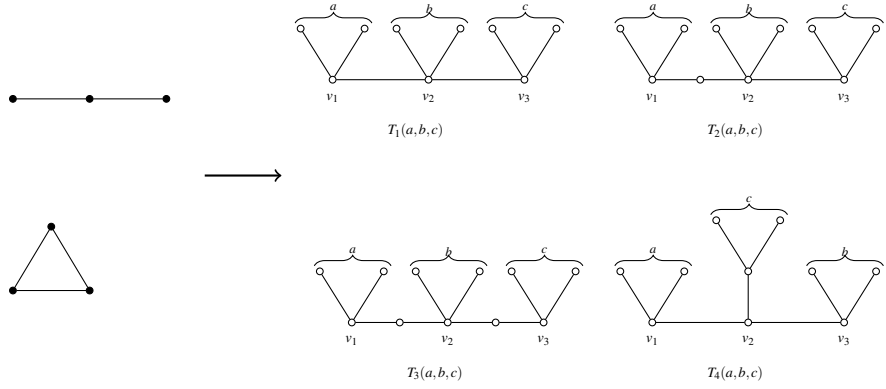

In order to prove Theorem \ref{matching3-min}, we need the following result from \cite{Xu2009}.
\begin{lem}[\cite{Xu2009}]\label{T-3}
	Let \( G \) be the spectral minimal graph in $\{ T_1(a,b,c), T_2(a,b,c), T_3(a,b,c), $\\$T_4(a,b,c) \}$, 
	where  \( a + b + c = n - 5\) ($n \geq 11$).
	Then \( G \cong T_3(a,b,c)\).
\end{lem}

Moreover, the following well known is used in our proof.

\begin{lem}[\cite{AW2011}]\label{characteristic}
	Let $GuvH$ be the graph obtained from graph $G$ and graph $H$ by adding an edge joining the vertex $u$ of $G$ and the vertex $v$ of $H$. The characteristic polynomial of $GuvH$ is given by
	\(
	\varphi(GuvH,x) = \varphi(G,x)\varphi(H,x) - \varphi(G-u,x)\varphi(H-v,x).
	\)
\end{lem}
\begin{proof}[\bf{Proof of Theorem \ref{matching3-min}}]
	By Theorem~\ref{thm-2.6}, \( T_3^* \in \mathcal{G}_{n,3} \) is a tree. For trees with matching number $3$, the quasi-adjacency graph can only be \( P_3 \) or \( C_3 \) by Corollary \ref{cor-2.1}.
	According to the mapping in Fig.~\ref{fig-main-tree-beta}, there are four types of graphs in \( \mathcal{T}_{n,3} \), as shown in Fig.~\ref{fig-7}.

By lemma \ref{T-3}, $T_3^* \cong T_3(a,b,c)$, where $a+b+c=n-5$. By symmetry, we may assume that \(a \geq c\). By Lemma \ref{average-min}, \(c+1 \geq a \geq c\).
	
\begin{claim} \label{c>=b+1}
$c \geq b+1$
\end{claim}
\begin{proof}
Otherwise, suppose \( b \geq c \). 
Notice that  \( T_3(a,b-1,c+1) \) and \( T_3(a,b,c) \) always have \( K_{1,a+1} \) as an induced subgraph, from Lemma \ref{propersubgraph} we have
\(
\rho(K_{1,a+1}) < \rho(T_3(a,b-1,c+1))\) and \(\rho(K_{1,a+1}) < \rho(T_3(a,b,c))\).

Based on the characteristic polynomial of the star graph and Lemma~\ref{characteristic}, we obtain
the characteristic polynomial of \(T_3(a,b,c)\) is
$\varphi(T_3(a,b,c),x)
= x^{\,a+b+c-1}\!\big[x^{6}-(a+b+c+4)x^{4}
		+\big(3a+2b+ab+c(a+b+3)+3\big)x^{2}
		-\big(a+b+ab+c(2a+b+ab+1)\big)
		\big]$.
Thus,
		$\varphi(T_3(a,b-1,c+1),x) - \varphi(T_3(a,b,c),x)
		= x^{\,a+b+c-1}\!\big[-ab+ac+(b-c)x^2-b+c+1\big]
		= x^{\,a+b+c-1}\!\big[(b-c)(x^2-(a+1))+1\big]> 0$ for \( x > \rho(K_{1,a+1}) \). 
It follows that $\rho(T_3(a,b-1,c+1)) < \rho(T_3(a,b,c))$, a contradiction.
\end{proof}

Note that \( n \ge 11 \). Then \( s = \left\lfloor \frac{n-5}{3} \right\rfloor \ge 2 \). 
Next we will distinguish three cases to discuss. 
	
\textbf{Case (i)} If \( n - 5 = 3s \). Since \( T_3(a,b,c) \) has proper subgraphs 
\( K_{1,a+1} \), \( K_{1,b+2} \), and \( K_{1,c+1} \), we have
\(
s+1 + \sqrt{3} = \rho^{2}(T_3(s+1,\, s-2,\, s+1)\ge \rho^{2}(T_3(a,b,c))> \max\{a+1,\; b+2,\; c+1\},
\)
which induces that \( a,c \le s + 1 \) and \( b \le s \). 
By Claim \ref{c>=b+1} and $a \geq c$, we have $T_3^* \cong T_3(s+1,\, s-2,\, s+1)$ or $T_3(s+1,\, s-1,\, s)$. 
Since $\varphi(T_3(s+1,\, s-1,\, s),x)=s+1 + \sqrt{3})=\sqrt{3}-2 < 0$, 
$\rho(T_3(s+1,\, s-1,\, s)) \geq s+1 + \sqrt{3} = \rho^{2}(T_3(s+1,\, s-2,\, s+1)$. Thus, we have $T_3^* \cong T_3(s+1,\, s-2,\, s+1)$.

\textbf{Case (ii)} If \( n - 5 = 3s+1 \). Since \( T_3(a,b,c) \) has proper subgraphs 
\( K_{1,a+1} \), \( K_{1,b+2} \), and \( K_{1,c+1} \), we have
\(s+3 = \rho^{2}(T_3(s+1,\, s-1,\, s+1)\ge \rho^{2}(T_3(a,b,c))> \max\{a+1,\; b+2,\; c+1\},\)
which induces that \( a,c \le s + 1 \) and \( b \le s \). By Claim \ref{c>=b+1} and $a \geq c$, we have $T_3^* \cong T_3(s+1,\, s-1,\, s+1)$.
	
\textbf{Case (iii)} If \( n - 5 = 3s+2 \). 
Since \( T_3(a,b,c) \) has proper subgraphs 
\( K_{1,a+1} \), \( K_{1,b+2} \), and \( K_{1,c+1} \), we have
\(
s+2+ \sqrt{2} = \rho^{2}(T_3(s+1,\, s,\, s+1)\ge \rho^{2}(T_3(a,b,c))> \max\{a+1,\; b+2,\; c+1\},
\)
which induces that \( a,c \le s + 2 \) and \( b \le s+1 \).
By Claim \ref{c>=b+1} and $a \geq c$, we have $T_3^* \cong T_3(s+1,\, s,\, s+1)$ , $T_3(s+2,\, s-2,\, s+2)$ or $T_3(s+2,\, s-1,\, s+1)$. 
After calculation, 
$\rho^{2}(T_3(s+2,\, s-2,\, s+2)=\frac{2s + 3 \pm \sqrt{17}}{2}=s+3.5616 \geq \rho^{2}(T_3(s+2,\, s-1,\, s+1)=s+3.5321 \geq \rho^{2}(T_3(s+1,\, s,\, s+1)=s+2+ \sqrt{2}$. 
Therefore, $T_3^* \cong T_3(s+1,\, s,\, s+1)$.
	
This completes  the proof.
\end{proof}

\section{Proof of Theorem \ref{thm-1.3}}

We first obtain the candidates for minimizers in $\mathcal{G}_{n,4}$.

By Theorem~\ref{thm-2.6}, $T_4^*\in \mathcal{G}_{n,4}$  is a tree. Then we can get a dominating set $X$ of $T_4^*$ with $|X|=4$ which satisfies the condition (i) and (ii) by Theorem \ref{lem-2.8}.
Moreover,  the quasi-adjacency graph $T_X$ of $X$ are  block graphs by Corollary \ref{cor-2.1}.
In fact, all block graphs of order $4$ are shown in Fig. \ref{fig-3} and they are all quasi-adjacent graphs of order $4$.
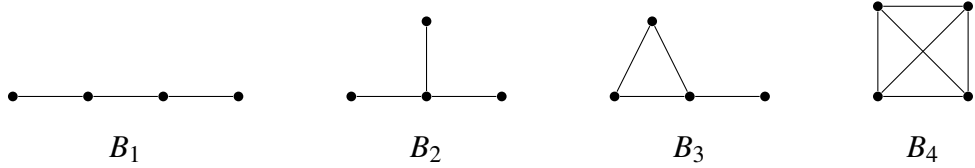
\begin{figure}[h]
\centering
\begin{tikzpicture}[every node/.style={circle,fill,inner sep=1.3pt}, x=1cm, y=1cm]

\begin{scope}[shift={(-1,0)}]
\node (b1) at (0,0) {};
\node (b2) at (1,0) {};
\node (b3) at (2,0) {};
\node (b4) at (3,0) {};
\draw (b1)--(b2)--(b3)--(b4);
\node[draw=none,fill=none] at (1.5,-0.7) {$B_1$};
\end{scope}

\begin{scope}[shift={(3.5,0)}]   
\node (b1) at (0,0) {};
\node (b2) at (1,0) {};
\node (b3) at (2,0) {};
\node (b4) at (1,1) {};
\draw (b1)--(b2)--(b3);
\draw (b2)--(b4);
\node[draw=none,fill=none] at (1,-0.7) {$B_2$};
\end{scope}

\begin{scope}[shift={(7,0)}]   
\node (b1) at (0,0) {};
\node (b2) at (1,0) {};
\node (b3) at (0.5,1) {};
\node (b4) at (2,0) {};
\draw (b1)--(b2)--(b3)--(b1);
\draw (b2)--(b4);
\node[draw=none,fill=none] at (1,-0.7) {$B_3$};
\end{scope}

\begin{scope}[shift={(10.5,0)}]   
\node (b1) at (0,0) {};
\node (b2) at (1.2,0) {};
\node (b3) at (1.2,1.2) {};
\node (b4) at (0,1.2) {};
\draw (b1)--(b2)--(b3)--(b4)--(b1);
\draw (b1)--(b3);
\draw (b2)--(b4);
\node[draw=none,fill=none] at (0.6,-0.7) {$B_4$};
\end{scope}

\end{tikzpicture}
\vspace*{-.7cm}
\caption{All block graphs of order 4.}\label{fig-3}
\end{figure}

If \( T_X \cong B_1 \), then the original graph 
\[
T_4^* \in \{ K_1(a,b,c,d), K_2(a,b,c,d), \cdots, K_6(a,b,c,d) \}
\]
(see Fig. \ref{fig-main-tree-4}). 
We call these trees \textit{the first type of trees} with a matching number of $4$.

\begin{figure}[h]
\centering
\begin{tikzpicture}[scale=1,
  every node/.style={circle,draw,inner sep=1pt}]
\begin{scope}[shift={(-2,0)}]

\node[label=below:{{\tiny $v_1$}}] (v1) at (0,0) {};
\node[label=below:{{\tiny $v_2$}}] (v2) at (1,0) {};
\node[label=below:{{\tiny $v_3$}}] (v3) at (2,0) {};
\node[label=below:{{\tiny $v_4$}}] (v4) at (3,0) {};
\draw (v1) --  (v2) -- (v3) -- (v4);

\node (v1l) at (-0.25,.75) {};
\node (v1r) at (.25,0.75)  {};
\draw (v1) -- (v1l) (v1) -- (v1r);
 \draw [decorate,decoration={brace,mirror,amplitude=4pt},xshift=0.4pt,yshift=-0.4pt]
(.35,.85) -- (-.35,0.85) node [black,draw=none,midway,yshift=.3cm]{{\tiny $a$}};

\node (v1l) at (.75,.75) {};
\node (v1r) at (1.25,0.75)  {};
\draw (v2) -- (v1l) (v2)-- (v1r);
\draw [decorate,decoration={brace,mirror,amplitude=4pt},xshift=0.4pt,yshift=-0.4pt]
(1.35,.85) -- (.65,0.85) node [black,draw=none,midway,yshift=.3cm]{{\tiny $b$}};

\node (v1l) at (1.75,.75) {};
\node (v1r) at (2.25,0.75)  {};
\draw (v3) -- (v1l) (v3)-- (v1r);
\draw [decorate,decoration={brace,mirror,amplitude=4pt},xshift=0.4pt,yshift=-0.4pt]
(2.35,.85) -- (1.65,0.85) node [black,draw=none,midway,yshift=.3cm]{{\tiny $c$}};

\node (v1l) at (2.75,.75) {};
\node (v1r) at (3.25,0.75)  {};
\draw (v4) -- (v1l) (v4)-- (v1r);
\draw [decorate,decoration={brace,mirror,amplitude=4pt},xshift=0.4pt,yshift=-0.4pt]
(3.35,.85) -- (2.65,0.85) node [black,draw=none,midway,yshift=.3cm]{{\tiny $d$}};

\node[draw=none,fill=none,font=\tiny] at (1.5,-0.7) {$K_1(a,b,c,d)$};

\end{scope}
\begin{scope}[shift={(3,0)}]

\node[label=below:{{\tiny $v_1$}}] (v1) at (0,0) {};
\node[label=below:{{\tiny $v_2$}}] (v2) at (1,0) {};
\node[label=below:{{\tiny $v_3$}}] (v3) at (2,0) {};
\node[label=below:{{\tiny $v_4$}}] (v4) at (3,0) {};
\node (v) at (.5,0){};
\draw (v1) -- (v) -- (v2) -- (v3) -- (v4);

\node (v1l) at (-0.25,.75) {};
\node (v1r) at (.25,0.75)  {};
\draw (v1) -- (v1l) (v1) -- (v1r);
 \draw [decorate,decoration={brace,mirror,amplitude=4pt},xshift=0.4pt,yshift=-0.4pt]
(.35,.85) -- (-.35,0.85) node [black,draw=none,midway,yshift=.3cm]{{\tiny $a$}};

\node (v1l) at (.75,.75) {};
\node (v1r) at (1.25,0.75)  {};
\draw (v2) -- (v1l) (v2)-- (v1r);
\draw [decorate,decoration={brace,mirror,amplitude=4pt},xshift=0.4pt,yshift=-0.4pt]
(1.35,.85) -- (.65,0.85) node [black,draw=none,midway,yshift=.3cm]{{\tiny $b$}};

\node (v1l) at (1.75,.75) {};
\node (v1r) at (2.25,0.75)  {};
\draw (v3) -- (v1l) (v3)-- (v1r);
\draw [decorate,decoration={brace,mirror,amplitude=4pt},xshift=0.4pt,yshift=-0.4pt]
(2.35,.85) -- (1.65,0.85) node [black,draw=none,midway,yshift=.3cm]{{\tiny $c$}};

\node (v1l) at (2.75,.75) {};
\node (v1r) at (3.25,0.75)  {};
\draw (v4) -- (v1l) (v4)-- (v1r);
\draw [decorate,decoration={brace,mirror,amplitude=4pt},xshift=0.4pt,yshift=-0.4pt]
(3.35,.85) -- (2.65,0.85) node [black,draw=none,midway,yshift=.3cm]{{\tiny $d$}};

\node[draw=none,fill=none,font=\tiny] at (1.5,-0.7) {$K_2(a,b,c,d)$};
\end{scope}
\begin{scope}[shift={(8,0)}]

\node[label=below:{{\tiny $v_1$}}] (v1) at (0,0) {};
\node[label=below:{{\tiny $v_2$}}] (v2) at (1,0) {};
\node[label=below:{{\tiny $v_3$}}] (v3) at (2,0) {};
\node[label=below:{{\tiny $v_4$}}] (v4) at (3,0) {};
\node (v) at (1.5,0){};
\draw (v1) -- (v2) -- (v)-- (v3) -- (v4);

\node (v1l) at (-0.25,.75) {};
\node (v1r) at (.25,0.75)  {};
\draw (v1) -- (v1l) (v1) -- (v1r);
\draw [decorate,decoration={brace,mirror,amplitude=4pt},xshift=0.4pt,yshift=-0.4pt]
(.35,.85) -- (-.35,0.85) node [black,draw=none,midway,yshift=.3cm]{{\tiny $a$}};

\node (v1l) at (.75,.75) {};
\node (v1r) at (1.25,0.75)  {};
\draw (v2) -- (v1l) (v2)-- (v1r);
\draw [decorate,decoration={brace,mirror,amplitude=4pt},xshift=0.4pt,yshift=-0.4pt]
(1.35,.85) -- (.65,0.85) node [black,draw=none,midway,yshift=.3cm]{{\tiny $b$}};

\node (v1l) at (1.75,.75) {};
\node (v1r) at (2.25,0.75)  {};
\draw (v3) -- (v1l) (v3)-- (v1r);
\draw [decorate,decoration={brace,mirror,amplitude=4pt},xshift=0.4pt,yshift=-0.4pt]
(2.35,.85) -- (1.65,0.85) node [black,draw=none,midway,yshift=.3cm]{{\tiny $c$}};

\node (v1l) at (2.75,.75) {};
\node (v1r) at (3.25,0.75)  {};
\draw (v4) -- (v1l) (v4)-- (v1r);
\draw [decorate,decoration={brace,mirror,amplitude=4pt},xshift=0.4pt,yshift=-0.4pt]
(3.35,.85) -- (2.65,0.85) node [black,draw=none,midway,yshift=.3cm]{{\tiny $d$}};

\node[draw=none,fill=none,font=\tiny] at (1.5,-0.7) {$K_3(a,b,c,d)$};
\end{scope}
\begin{scope}[shift={(-2,-2.5)}]

\node[label=below:{{\tiny $v_1$}}] (v1) at (0,0) {};
\node[label=below:{{\tiny $v_2$}}] (v2) at (1,0) {};
\node[label=below:{{\tiny $v_3$}}] (v3) at (2,0) {};
\node[label=below:{{\tiny $v_4$}}] (v4) at (3,0) {};
\node (v) at (1.5,0){};\node (u) at (.5,0) {};
\draw (v1) --(u) -- (v2) -- (v)-- (v3) -- (v4);

\node (v1l) at (-0.25,.75) {};
\node (v1r) at (.25,0.75)  {};
\draw (v1) -- (v1l) (v1) -- (v1r);
\draw [decorate,decoration={brace,mirror,amplitude=4pt},xshift=0.4pt,yshift=-0.4pt]
(.35,.85) -- (-.35,0.85) node [black,draw=none,midway,yshift=.3cm]{{\tiny $a$}};

\node (v1l) at (.75,.75) {};
\node (v1r) at (1.25,0.75)  {};
\draw (v2) -- (v1l) (v2)-- (v1r);
\draw [decorate,decoration={brace,mirror,amplitude=4pt},xshift=0.4pt,yshift=-0.4pt]
(1.35,.85) -- (.65,0.85) node [black,draw=none,midway,yshift=.3cm]{{\tiny $b$}};

\node (v1l) at (1.75,.75) {};
\node (v1r) at (2.25,0.75)  {};
\draw (v3) -- (v1l) (v3)-- (v1r);
\draw [decorate,decoration={brace,mirror,amplitude=4pt},xshift=0.4pt,yshift=-0.4pt]
(2.35,.85) -- (1.65,0.85) node [black,draw=none,midway,yshift=.3cm]{{\tiny $c$}};

\node (v1l) at (2.75,.75) {};
\node (v1r) at (3.25,0.75)  {};
\draw (v4) -- (v1l) (v4)-- (v1r);
\draw [decorate,decoration={brace,mirror,amplitude=4pt},xshift=0.4pt,yshift=-0.4pt]
(3.35,.85) -- (2.65,0.85) node [black,draw=none,midway,yshift=.3cm]{{\tiny $d$}};

\node[draw=none,fill=none,font=\tiny] at (1.5,-0.7) {$K_4(a,b,c,d)$};
\end{scope}
\begin{scope}[shift={(3,-2.5)}]

\node[label=below:{{\tiny $v_1$}}] (v1) at (0,0) {};
\node[label=below:{{\tiny $v_2$}}] (v2) at (1,0) {};
\node[label=below:{{\tiny $v_3$}}] (v3) at (2,0) {};
\node[label=below:{{\tiny $v_4$}}] (v4) at (3,0) {};
\node (v) at (2.5,0){};\node (u) at (.5,0) {};
\draw (v1) --(u) -- (v2) --  (v3)-- (v)-- (v4);

\node (v1l) at (-0.25,.75) {};
\node (v1r) at (.25,0.75)  {};
\draw (v1) -- (v1l) (v1) -- (v1r);
\draw [decorate,decoration={brace,mirror,amplitude=4pt},xshift=0.4pt,yshift=-0.4pt]
(.35,.85) -- (-.35,0.85) node [black,draw=none,midway,yshift=.3cm]{{\tiny $a$}};

\node (v1l) at (.75,.75) {};
\node (v1r) at (1.25,0.75)  {};
\draw (v2) -- (v1l) (v2)-- (v1r);
\draw [decorate,decoration={brace,mirror,amplitude=4pt},xshift=0.4pt,yshift=-0.4pt]
(1.35,.85) -- (.65,0.85) node [black,draw=none,midway,yshift=.3cm]{{\tiny $b$}};

\node (v1l) at (1.75,.75) {};
\node (v1r) at (2.25,0.75)  {};
\draw (v3) -- (v1l) (v3)-- (v1r);
\draw [decorate,decoration={brace,mirror,amplitude=4pt},xshift=0.4pt,yshift=-0.4pt]
(2.35,.85) -- (1.65,0.85) node [black,draw=none,midway,yshift=.3cm]{{\tiny $c$}};

\node (v1l) at (2.75,.75) {};
\node (v1r) at (3.25,0.75)  {};
\draw (v4) -- (v1l) (v4)-- (v1r);
\draw [decorate,decoration={brace,mirror,amplitude=4pt},xshift=0.4pt,yshift=-0.4pt]
(3.35,.85) -- (2.65,0.85) node [black,draw=none,midway,yshift=.3cm]{{\tiny $d$}};

\node[draw=none,fill=none,font=\tiny] at (1.5,-0.7) {$K_5(a,b,c,d)$};
\end{scope}
\begin{scope}[shift={(8,-2.5)}]

\node[label=below:{{\tiny $v_1$}}] (v1) at (0,0) {};
\node[label=below:{{\tiny $v_2$}}] (v2) at (1,0) {};
\node[label=below:{{\tiny $v_3$}}] (v3) at (2,0) {};
\node[label=below:{{\tiny $v_4$}}] (v4) at (3,0) {};
\node (v) at (2.5,0){};\node[label=below:{{\tiny $u_1$}}] (u) at (.5,0) {};\node[label=below:{{\tiny $u_2$}}] (w) at (1.5,0) {};
\draw (v1) --(u) -- (v2) -- (w)-- (v3)-- (v)-- (v4);

\node (v1l) at (-0.25,.75) {};
\node (v1r) at (.25,0.75)  {};
\draw (v1) -- (v1l) (v1) -- (v1r);
\draw [decorate,decoration={brace,mirror,amplitude=4pt},xshift=0.4pt,yshift=-0.4pt]
(.35,.85) -- (-.35,0.85) node [black,draw=none,midway,yshift=.3cm]{{\tiny $a$}};

\node (v1l) at (.75,.75) {};
\node (v1r) at (1.25,0.75)  {};
\draw (v2) -- (v1l) (v2)-- (v1r);
\draw [decorate,decoration={brace,mirror,amplitude=4pt},xshift=0.4pt,yshift=-0.4pt]
(1.35,.85) -- (.65,0.85) node [black,draw=none,midway,yshift=.3cm]{{\tiny $b$}};

\node (v1l) at (1.75,.75) {};
\node (v1r) at (2.25,0.75)  {};
\draw (v3) -- (v1l) (v3)-- (v1r);
\draw [decorate,decoration={brace,mirror,amplitude=4pt},xshift=0.4pt,yshift=-0.4pt]
(2.35,.85) -- (1.65,0.85) node [black,draw=none,midway,yshift=.3cm]{{\tiny $c$}};

\node (v1l) at (2.75,.75) {};
\node (v1r) at (3.25,0.75)  {};
\draw (v4) -- (v1l) (v4)-- (v1r);
\draw [decorate,decoration={brace,mirror,amplitude=4pt},xshift=0.4pt,yshift=-0.4pt]
(3.35,.85) -- (2.65,0.85) node [black,draw=none,midway,yshift=.3cm]{{\tiny $d$}};

\node[draw=none,fill=none,font=\tiny] at (1.5,-0.7) {$K_6(a,b,c,d)$};
\end{scope}
\end{tikzpicture}
\vspace*{-1cm}
\caption{The first type of trees with matching number 4.}\label{fig-main-tree-4}
\end{figure}
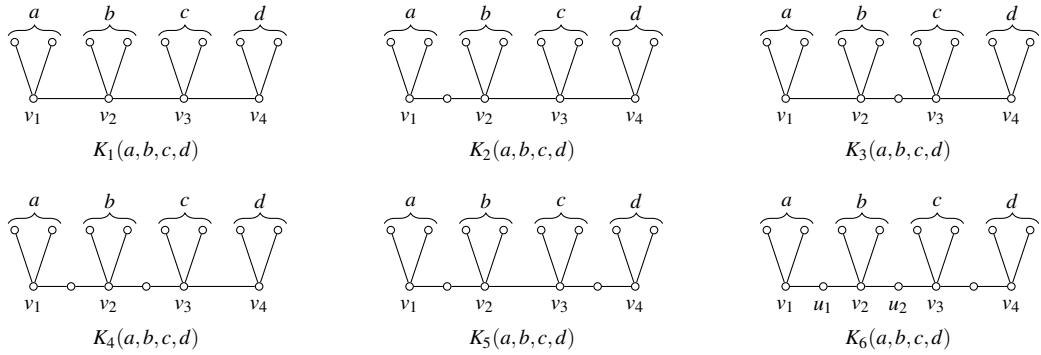

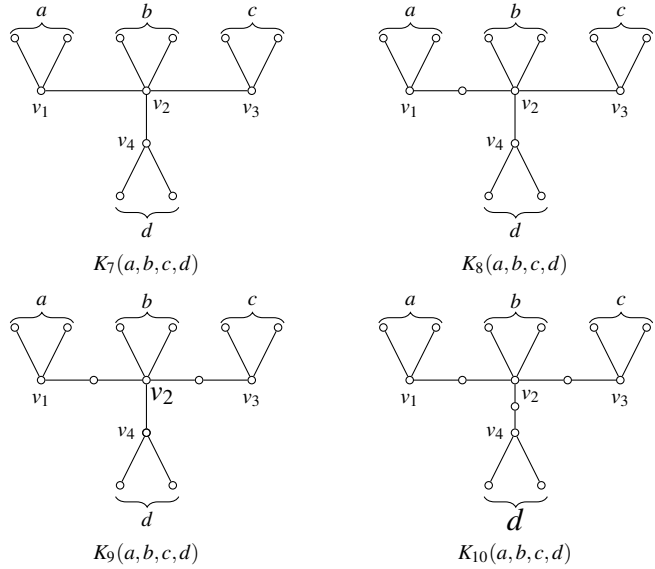
\begin{figure}[h!]
  \centering
\begin{tikzpicture}[scale=.7,
  every node/.style={circle,draw,inner sep=1pt}]
\begin{scope}[shift={(-2,0)}]
\foreach \i in {1,3,5}{
     \node[draw,circle] (\i) at (\i, 0){};}
\foreach \from/\to in {1/3,3/5}{
\draw (\from) -- (\to);}

\node[draw,circle] (a) at (.5, 1){};
\node[draw,circle] (b) at (1.5, 1){};
\draw (a)--(1)--(b);
\draw [decorate,decoration={brace,mirror,amplitude=4pt},xshift=0.4pt,yshift=-0.4pt]
(1.6,1.1) -- (.4,1.1) node [draw=none,black,midway,yshift=.3cm]{\tiny $a$};
\node[draw,circle] (a) at (2.5, 1){};
\node[draw,circle] (b) at (3.5, 1){};
\draw (a)--(3)--(b);
\draw [decorate,decoration={brace,mirror,amplitude=4pt},xshift=0.4pt,yshift=-0.4pt]
(3.6,1.1) -- (2.4,1.1) node[draw=none,black,midway,yshift=.3cm]{\tiny $b$};

\node[draw,circle] (a) at (4.5, 1){};
\node[draw,circle] (b) at (5.5, 1){};
\draw (a)--(5)--(b);
\draw [decorate,decoration={brace,mirror,amplitude=4pt},xshift=0.4pt,yshift=-0.4pt]
(5.6,1.1) -- (4.4,1.1) node [draw=none,black,midway,yshift=.3cm]{\tiny $c$};

\node[draw,circle,label=left:\tiny $v_4$] (c) at (3, -1){};
\node[draw,circle] (a) at (2.5,-2){};
\node[draw,circle] (b) at (3.5,-2){};
\draw (a)--(c)--(b) (3)--(c); 
\draw [decorate,decoration={brace,mirror,amplitude=4pt},xshift=0.4pt,yshift=-0.4pt]
(2.4,-2.2) -- (3.6,-2.2) node [draw=none,black,midway,yshift=-.3cm]{\tiny $d$};
\node[draw=none,fill=none,font=\small] at (3,-3.3) {\tiny $K_7(a,b,c,d)$};
\node[draw=none,fill=none,font=\tiny] at (3.3,-0.3) {$v_2$};
\node[draw=none,label=below:{ \tiny$v_1$}] at (1,0){};
\node[draw=none,label=below:{ \tiny $v_3$}] at (5,0){};
\end{scope}
\begin{scope}[shift={(5,0)}]
\foreach \i in {1,2,3,5}{
     \node[draw,circle] (\i) at (\i, 0){};}
\foreach \from/\to in {1/2,2/3,3/5}{
\draw (\from) -- (\to);}

\node[draw,circle] (a) at (.5, 1){};
\node[draw,circle] (b) at (1.5, 1){};
\draw (a)--(1)--(b);
\draw [decorate,decoration={brace,mirror,amplitude=4pt},xshift=0.4pt,yshift=-0.4pt]
(1.6,1.1) -- (.4,1.1) node [draw=none,black,midway,yshift=.3cm]{\tiny $a$};
\node[draw,circle] (a) at (2.5, 1){};
\node[draw,circle] (b) at (3.5, 1){};
\draw (a)--(3)--(b);
\draw [decorate,decoration={brace,mirror,amplitude=4pt},xshift=0.4pt,yshift=-0.4pt]
(3.6,1.1) -- (2.4,1.1) node [black,draw=none,midway,yshift=.3cm]{{\tiny $b$}};

\node[draw,circle] (a) at (4.5, 1){};
\node[draw,circle] (b) at (5.5, 1){};
\draw (a)--(5)--(b);
\draw [decorate,decoration={brace,mirror,amplitude=4pt},xshift=0.4pt,yshift=-0.4pt]
(5.6,1.1) -- (4.4,1.1) node [draw=none,black,midway,yshift=.3cm]{\tiny $c$};

\node[draw,circle,label=left:\tiny $v_4$] (c) at (3, -1){};
\node[draw,circle] (a) at (2.5,-2){};
\node[draw,circle] (b) at (3.5,-2){};
\draw (a)--(c)--(b) (3)--(c); 
\draw [decorate,decoration={brace,mirror,amplitude=4pt},xshift=0.4pt,yshift=-0.4pt]
(2.4,-2.2) -- (3.6,-2.2) node [draw=none,black,midway,yshift=-.3cm]{\tiny $d$};
\node[draw=none,fill=none,font=\tiny] at (3,-3.3) {$K_8(a,b,c,d)$};
\node[draw=none,fill=none,font=\tiny] at (3.3,-0.3) {$v_2$};
\node[draw=none,label=below:{ \tiny $v_1$}] at (1,0){};
\node[draw=none,label=below:{ \tiny $v_3$}] at (5,0){};
\end{scope}
\begin{scope}[shift={(-2,-5.5)}]
\foreach \i in {1,2,3,4,5}{
     \node[draw,circle] (\i) at (\i, 0){};}
\foreach \from/\to in {1/2,2/3,3/4,4/5}{
\draw (\from) -- (\to);}

\node[draw,circle] (a) at (.5, 1){};
\node[draw,circle] (b) at (1.5, 1){};
\draw (a)--(1)--(b);
\draw [decorate,decoration={brace,mirror,amplitude=4pt},xshift=0.4pt,yshift=-0.4pt]
(1.6,1.1) -- (.4,1.1) node [draw=none,black,midway,yshift=.3cm]{\tiny $a$};
\node[draw,circle] (a) at (2.5, 1){};
\node[draw,circle] (b) at (3.5, 1){};
\draw (a)--(3)--(b);
\draw [decorate,decoration={brace,mirror,amplitude=4pt},xshift=0.4pt,yshift=-0.4pt]
(3.6,1.1) -- (2.4,1.1) node [black,draw=none,midway,yshift=.3cm]{{\tiny $b$}};

\node[draw,circle] (a) at (4.5, 1){};
\node[draw,circle] (b) at (5.5, 1){};
\draw (a)--(5)--(b);
\draw [decorate,decoration={brace,mirror,amplitude=4pt},xshift=0.4pt,yshift=-0.4pt]
(5.6,1.1) -- (4.4,1.1) node [draw=none,black,midway,yshift=.3cm]{\tiny $c$};

\node[draw,circle] (c) at (3, -1){};
\node[draw,circle] (a) at (2.5,-2){};
\node[draw,circle] (b) at (3.5,-2){};
\draw (a)--(c)--(b); \draw (3)--(c);
\draw [decorate,decoration={brace,mirror,amplitude=4pt},xshift=0.4pt,yshift=-0.4pt]
(2.4,-2.2) -- (3.6,-2.2) node [draw=none,black,midway,yshift=-.3cm]{\tiny $d$};
\node[draw=none,fill=none,font=\tiny] at (3,-3.3) {$K_9(a,b,c,d)$};
\node[draw=none,fill=none,font=\small] at (3.3,-0.3) {$v_2$};
\node[draw=none,label=below:{ \tiny $v_1$}] at (1,0){};
\node[draw=none,label=below:{\tiny $v_3$}] at (5,0){};
\node[draw,circle,label=left:\tiny $v_4$] (c) at (3, -1){};
\end{scope}
\begin{scope}[shift={(5,-5.5)}]
\foreach \i in {1,2,3,4,5}{
     \node[draw,circle] (\i) at (\i, 0){};}
\foreach \from/\to in {1/2,2/3,3/4,4/5}{
\draw (\from) -- (\to);}

\node[draw,circle] (a) at (.5, 1){};
\node[draw,circle] (b) at (1.5, 1){};
\draw (a)--(1)--(b);
\draw [decorate,decoration={brace,mirror,amplitude=4pt},xshift=0.4pt,yshift=-0.4pt]
(1.6,1.1) -- (.4,1.1) node [draw=none,black,midway,yshift=.3cm]{\tiny $a$};
\node[draw,circle] (a) at (2.5, 1){};
\node[draw,circle] (b) at (3.5, 1){};
\draw (a)--(3)--(b);
\draw [decorate,decoration={brace,mirror,amplitude=4pt},xshift=0.4pt,yshift=-0.4pt]
(3.6,1.1) -- (2.4,1.1) node [black,draw=none,midway,yshift=.3cm]{{\tiny $b$}};

\node[draw,circle] (a) at (4.5, 1){};
\node[draw,circle] (b) at (5.5, 1){};
\draw (a)--(5)--(b);
\draw [decorate,decoration={brace,mirror,amplitude=4pt},xshift=0.4pt,yshift=-0.4pt]
(5.6,1.1) -- (4.4,1.1) node [draw=none,black,midway,yshift=.3cm]{\tiny $c$};

\node[draw,circle] (c) at (3, -.5){};
\node[draw,circle,label=left:\tiny $v_4$] (d) at (3, -1){};
\node[draw,circle] (a) at (2.5,-2){};
\node[draw,circle] (b) at (3.5,-2){};
\draw (a)--(d)--(b); \draw (3)--(c)--(d);
\draw [decorate,decoration={brace,mirror,amplitude=4pt},xshift=0.4pt,yshift=-0.4pt]
(2.4,-2.2) -- (3.6,-2.2) node [draw=none,black,midway,yshift=-.3cm]{$d$};
\node[draw=none,fill=none,font=\tiny] at (3,-3.3) {$K_{10}(a,b,c,d)$};
\node[draw=none,fill=none,font=\tiny] at (3.3,-0.3) {$v_2$};
\node[draw=none,label=below:{\tiny $v_1$}] at (1,0){};
\node[draw=none,label=below:{\tiny $v_3$}] at (5,0){};
\end{scope}
\end{tikzpicture}
\vspace*{-1cm}
	\caption{The second type of trees with a matching number of $4$.}\label{fig-main-tree-5}
\end{figure}

If \( T_X \cong B_2 \), then the original graph 
\[
T_4^* \in \{ K_7(a,b,c,d), K_8(a,b,c,d), K_9(a,b,c,d), K_{10}(a,b,c,d) \}
\]
(see Fig. \ref{fig-main-tree-5}). We call these trees \textit{the second type of trees} with a matching number of $4$.

If \( T_X \cong B_3 \), then the original graph 
\[
T_4^* \in \{ K_{11}(a,b,c,d), K_{12}(a,b,c,d) \}\ \ \mbox{(see Fig. \ref{fig-main-tree-6}).}
\]
We call these trees \textit{the third type of trees} with a matching number of $4$.

If \( T_X \cong B_4 \), then the original graph \( T_4^* \cong K_{13}(a,b,c,d)  \) (see Fig. \ref{fig-main-tree-6}).
We call these trees \textit{the fourth type of trees} with a matching number of $4$.

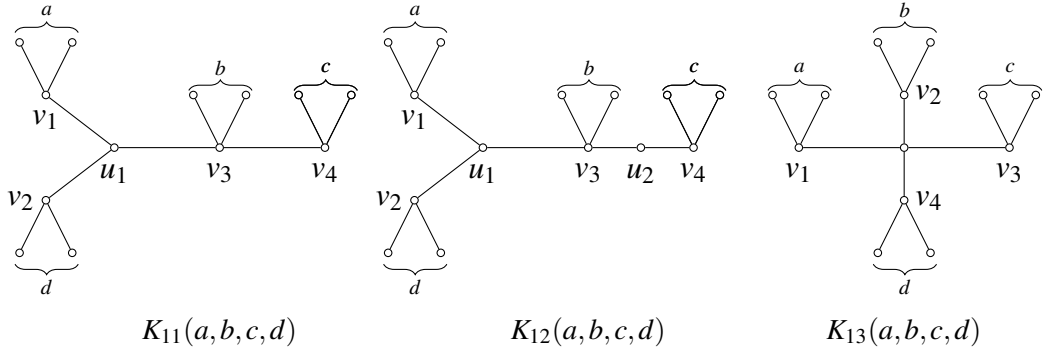
\begin{figure}[h!]
  \centering
\begin{tikzpicture}[scale=0.7, every node/.style={draw,circle, inner sep=1pt}]
\begin{scope}[shift={(-2,0)}]
  
\node[label=below:$u_1$] (u1) at (1,0) {};
\node[label=below:{ $v_3$}] (v3) at (3,0){};
\node[label=below:{ $v_4$}] (v4) at (5,0){};
\node[label=below:{ $v_1$}] (v1) at (-.3,1){};
\node[label=left:{ $v_2$}] (v2) at (-.3,-1){};
\draw (v1)--(u1)--(v3)--(v4) (v2)--(u1);

\node[draw,circle] (a) at (-.8, 2){};
\node[draw,circle] (b) at (.2, 2){};
\draw (a)--(v1)--(b);
\draw [decorate,decoration={brace,mirror,amplitude=4pt},xshift=0.4pt,yshift=-0.4pt]
(.3,2.2) -- (-.9,2.2) node [draw=none,black,midway,yshift=.3cm]{\tiny $a$};

\node[draw,circle] (a) at (-.8, -2){};
\node[draw,circle] (b) at (.2, -2){};
\draw (a)--(v2)--(b);
\draw [decorate,decoration={brace,mirror,amplitude=4pt},xshift=0.4pt,yshift=-0.4pt]
 (-.9,-2.2) --(.3,-2.2) node [draw=none,black,midway,yshift=-.3cm]{\tiny $d$};
\node[draw,circle] (a) at (4.5, 1){};
\node[draw,circle] (b) at (5.5, 1){};
\draw (a)--(v4)--(b);
\draw [decorate,decoration={brace,mirror,amplitude=4pt},xshift=0.4pt,yshift=-0.4pt]
(5.6,1.1) -- (4.4,1.1) node [draw=none,black,midway,yshift=.3cm]{\tiny $c$};

\node[draw,circle] (a) at (3.5, 1){};
\node[draw,circle] (b) at (2.5, 1){};
\draw (a)--(v3)--(b);
\draw [decorate,decoration={brace,mirror,amplitude=4pt},xshift=0.4pt,yshift=-0.4pt]
(3.6,1.1) -- (2.4,1.1) node [draw=none,black,midway,yshift=.3cm]{\tiny $b$};

\node[draw,circle] (a) at (4.5, 1){};
\node[draw,circle] (b) at (5.5, 1){};
\draw (a)--(v4)--(b);
\draw [decorate,decoration={brace,mirror,amplitude=4pt},xshift=0.4pt,yshift=-0.4pt]
(5.6,1.1) -- (4.4,1.1) node [draw=none,black,midway,yshift=.3cm]{\tiny $c$};
\node[draw=none,fill=none,font=\small] at (3,-3.5) {$K_{11}(a,b,c,d)$};
\end{scope}
\begin{scope}[shift={(5,0)}]
  
\node[label=below:$u_1$] (u1) at (1,0) {};
\node[label=below:{ $v_3$}] (v3) at (3,0){};
\node[label=below:{ $v_4$}] (v4) at (5,0){};
\node[label=below:{ $v_1$}] (v1) at (-.3,1){};
\node[label=left:{ $v_2$}] (v2) at (-.3,-1){};
\node[label=below:{ $u_2$}] (u2) at (4,0){};
\draw (v1)--(u1)--(v3)--(u2)--(v4) (v2)--(u1);

\node[draw,circle] (a) at (-.8, 2){};
\node[draw,circle] (b) at (.2, 2){};
\draw (a)--(v1)--(b);
\draw [decorate,decoration={brace,mirror,amplitude=4pt},xshift=0.4pt,yshift=-0.4pt]
(.3,2.2) -- (-.9,2.2) node [draw=none,black,midway,yshift=.3cm]{\tiny $a$};

\node[draw,circle] (a) at (-.8, -2){};
\node[draw,circle] (b) at (.2, -2){};
\draw (a)--(v2)--(b);
\draw [decorate,decoration={brace,mirror,amplitude=4pt},xshift=0.4pt,yshift=-0.4pt]
 (-.9,-2.2) --(.3,-2.2) node [draw=none,black,midway,yshift=-.3cm]{\tiny $d$};
\node[draw,circle] (a) at (4.5, 1){};
\node[draw,circle] (b) at (5.5, 1){};
\draw (a)--(v4)--(b);
\draw [decorate,decoration={brace,mirror,amplitude=4pt},xshift=0.4pt,yshift=-0.4pt]
(5.6,1.1) -- (4.4,1.1) node [draw=none,black,midway,yshift=.3cm]{\tiny $c$};

\node[draw,circle] (a) at (3.5, 1){};
\node[draw,circle] (b) at (2.5, 1){};
\draw (a)--(v3)--(b);
\draw [decorate,decoration={brace,mirror,amplitude=4pt},xshift=0.4pt,yshift=-0.4pt]
(3.6,1.1) -- (2.4,1.1) node [draw=none,black,midway,yshift=.3cm]{\tiny $b$};

\node[draw,circle] (a) at (4.5, 1){};
\node[draw,circle] (b) at (5.5, 1){};
\draw (a)--(v4)--(b);
\draw [decorate,decoration={brace,mirror,amplitude=4pt},xshift=0.4pt,yshift=-0.4pt]
(5.6,1.1) -- (4.4,1.1) node [draw=none,black,midway,yshift=.3cm]{\tiny $c$};
\node[draw=none,fill=none,font=\small] at (3,-3.5) {$K_{12}(a,b,c,d)$};
\begin{scope}[shift={(6,0)}]
  
\node[label=below:{ $v_3$}] (v3) at (5,0){};
\node[label=right:{ $v_4$}] (v4) at (3,-1){};
\node[label=below:{ $v_1$}] (v1) at (1,0){};
\node[label=right:{ $v_2$}] (v2) at (3,1){};
\node (u2) at (3,0){};
\draw (v1)--(u2)--(v3) (v2)--(u2)--(v4);

\node[draw,circle] (a) at (.5, 1){};
\node[draw,circle] (b) at (1.5, 1){};
\draw (a)--(v1)--(b);
\draw [decorate,decoration={brace,mirror,amplitude=4pt},xshift=0.4pt,yshift=-0.4pt]
(1.6,1.1) -- (.4,1.1) node [draw=none,black,midway,yshift=.3cm]{\tiny $a$};

\node[draw,circle] (a) at (3.5, -2){};
\node[draw,circle] (b) at (2.5, -2){};
\draw (a)--(v4)--(b);
\draw [decorate,decoration={brace,mirror,amplitude=4pt},xshift=0.4pt,yshift=-0.4pt]
 (2.4,-2.2)--(3.6,-2.2)  node [draw=none,black,midway,yshift=-.3cm]{\tiny $d$};

\node[draw,circle] (a) at (4.5, 1){};
\node[draw,circle] (b) at (5.5, 1){};
\draw (a)--(v3)--(b);
\draw [decorate,decoration={brace,mirror,amplitude=4pt},xshift=0.4pt,yshift=-0.4pt]
(5.6,1.1) -- (4.4,1.1) node [draw=none,black,midway,yshift=.3cm]{\tiny $c$};

\node[draw,circle] (a) at (3.5, 2){};
\node[draw,circle] (b) at (2.5, 2){};
\draw (a)--(v2)--(b);
\draw [decorate,decoration={brace,mirror,amplitude=4pt},xshift=0.4pt,yshift=-0.4pt]
(3.6,2.2) -- (2.4,2.2) node [draw=none,black,midway,yshift=.3cm]{\tiny $b$};
\node[draw=none,fill=none,font=\small] at (3,-3.5) {$K_{13}(a,b,c,d)$};
\end{scope}
\end{scope}
\end{tikzpicture}
\vspace*{-1cm}
	\caption{The third and the fourth type  trees with matching number of $4$.}\label{fig-main-tree-6}
\end{figure}

The classification of graphs provides a foundation for identifying the extremal graph with the minimum spectral radius. 
Building on this, our next step is to determine which graph among these candidates has the minimum spectral radius. We first prove that $T_4^* \ncong K_{12}(a,b,c,d)$ or $K_{13}(a,b,c,d)$.

Let \(X = \{v_1, v_2, v_3, v_4\}\) be the set of quasi-pendant vertices of \(T_4^*\), and let \(v^*\) be the vertex among \(v_1, v_2, v_3, v_4\) with the largest number of hanging leaves. 
Recall that $n\geq 19$.
By the pigeonhole principle, \(v^*\) has at least two leaves.

	In the remaining of this section, in order to characterize the minimizer in $\mathcal{G}_{n,4}$, we discard candidates. By assuming that a graph $G$ is minimizer, we make operations to find a graph $G'\in \mathcal{G}_{n,4}$, with the same number of vertices  such that $\rho(G') < \rho(G)$, contradicting the minimality. The main operation is the combination of Lemmas~\ref{subdivision} and \ref{propersubgraph} as already used in the previous section. But we also need the following result showing that if we make a larger pendant path larger, while making a smaller pendant path smaller, this decreases the spectral radius.

\begin{lem}[\cite{LiFeng1979}]\label{lem-2-1}
	Let $v$ be a vertex in a graph $G$ and suppose that two new paths 
	$P_{k+1}: vv_1v_2\cdots v_k$ and $P_{\ell+1}:uu_1u_2\cdots u_\ell$ ($k\geq \ell\geq 1$) are  attached to $G$ at $v$, respectively, to form a new graph $\rho(G_{k,\ell})$. Then
	$\rho(G_{k,\ell})>\rho(G_{k+1,\ell-1})$.
\end{lem}

In addition, the following result is also going to be used. 

\begin{lem}[\cite{WuXiao2005}]\label{lem-perron-entry-radius}
	Let $u$, $v$ be two distinct vertices of a connected graph $G$. Suppose $w_1,w_2,...,w_s$ $(s\ge1)$ are some vertices of $N_G(v)\setminus N_G(u)$ and $\mathbf{x}$ is perron vector of $G$. Let
	$G'=G-\{vw_i: i=1,2,...,s\}+\{uw_i: i=1,2,...,s\}$.
	If $x_u\ge x_v$, then $\rho(G)<\rho(G')$.
\end{lem}

\begin{lem}\label{lem-3.8}
	\(T_4^*\ncong K_1(a,b,c,d)\) for any integers \(a, b, c, d \geq 1\).
\end{lem}

\begin{proof}
By contradiction, assume that \(T_4^* \cong K_1(a,b,c,d)\) with \(a, b, c, d \geq 1\). 
	Construct a tree \(T_1\) by subdividing one internal edge of  $K_1(a,b,c,d)$ and deleting one pendant edge incident to $v^*$ from \(K_1(a,b,c,d)\). 
	Clearly,  \(T_1 \in \mathcal{G}_{n,4}\). 
	By Lemmas  \ref{subdivision} and \ref{propersubgraph}, we have
	\[
	\rho(T_1) \leq \rho(K_1(a,b,c,d)),
	\]
	which contradicts the minimality of \(\rho(T_4^*)\). 
	Hence, the lemma follows.
\end{proof}

\begin{lem}\label{lem-3.9}
	\(T_4^*\ncong K_2(a,b,c,d)\) for any integers \(a \geq 1\), \(b \geq 0\), \(c \geq 1\), \(d \geq 1\).
\end{lem}

\begin{proof}
	Suppose, by contradiction, that \(T_4^* \cong K_2(a,b,c,d)\) with \(a \geq 1\), \(b \geq 0\), \(c \geq 1\), \(d \geq 1\). 
	We consider the following cases:
	
	If \(a = 1\), \(b = 0\) and \(d = 1\), then let \(T_2\) be the tree obtained from \(K_2(a,b,c,d)\) by deleting one pendant edge incident to $v_3$ and attaching a new pendant edge to the only leaf adjacent to \(v_1\). Then \(T_2 \in \mathcal{T}_{n,4}\). By Lemma~\ref{lem-2-1},
	\(
	\rho(T_2) < \rho(K_2(a,b,c,d)),
	\)
	which also contradicts the minimality of \(\rho(T_4^*)\).
	
	If \(a \geq 2\) or  \(b \geq 1\) or \(d \geq 2\), then $T^*$ we notice that it  always has an internal edge. 
	Let \(T_3\) be the tree obtained from \(K_2(a,b,c,d)\) by subdividing one internal edge ($v_2v_3$ or $v_3v_4$) and deleting one pendant edge incident to \(v^*\).
	Clearly, \(T_3 \in \mathcal{T}_{n,4}\). By Lemmas~\ref{subdivision} and~\ref{propersubgraph}, we have
	\(
	\rho(T_3) < \rho(K_2(a,b,c,d)),
	\)
	contradicting the minimality of \(\rho(T_4^*)\).
\end{proof}

It is worth noting that there may exist some isomorphic graphs among \(K_1\) to \(K_{13}\).  For instance, $K_3(a,b,c,d)$ can also be expressed as $K_2(a,b,c,d)$ or $K_2(a,c,b,d)$ when $b=0$ or $c=0$. To avoid redundant consideration, we have specified the parameter ranges in the following results.

\begin{lem}\label{lem-3.10}
	\(T_4^*\ncong K_3(a,b,c,d)\) for any integers \(a, b, c, d \geq 1\).
\end{lem}

\begin{proof}
	Assume that \(T_4^* \cong K_3(a,b,c,d)\) with \(a, b, c, d \geq 1\). 
	
	If $a \geq 2$ or $d\geq 2$, let \(T_4\) be the tree obtained from \(K_3(a,b,c,d)\) by subdividing $v_1v_2$ (if $a \geq 2)$ or $v_3v_4$ (if $d \geq 2)$ and deleting one pendant edge incident to \(v^*\). 
	Clearly, \(T_4 \in \mathcal{T}_{n,4}\). 
	By Lemmas  \ref{subdivision} and \ref{propersubgraph}, we have
	\(
	\rho(T_4) < \rho(K_3(a,b,c,d)),
	\)
	which contradicts the minimality of \(\rho(T_4^*)\). Otherwise, when $a=1$ and $d=1$, let \(T_4\) be the tree obtained from \(K_3(a,b,c,d)\) by subdividing one edge of $P_{v_2v_3}$ and deleting one pendant edge incident to \(v_1\). 
	Clearly, \(T_4 \in \mathcal{T}_{n,4}\). 
	By Lemmas  \ref{subdivision} and \ref{propersubgraph}, we also have
	\(
	\rho(T_4) < \rho(K_3(a,b,c,d)),
	\)
	which contradicts the minimality of \(\rho(T_4^*)\).
\end{proof}

\begin{lem}\label{lem-3.11}
	$T_4^* \not\cong K_4(a,b,c,d)$ for any integers \(a \ge 1\), \(b \ge 0\), \(c \ge 0\), and \(d \ge 1\).
\end{lem}
\begin{proof}
	Assume that $T_4^* \cong K_4(a,b,c,d)$ with $a \ge 1$, $b \ge 0$, $c \ge 0$, and $d \ge 1$.
	
For $d \ge 2$ and $(a,b,c) \neq (1,0,0)$,  let $T_5$ be the tree obtained from $K_4(a,b,c,d)$ by subdividing the edge $v_3v_4$ and deleting one pendant edge incident to $v^*$.
	Clearly, $T_5 \in \mathcal{G}_{n,4}$.
	By Lemmas \ref{subdivision} and \ref{propersubgraph}, $\rho(T_5) < \rho(K_4(a,b,c,d))$, contradicting the minimality of $\rho(T_4^*)$.
	
	If $d \ge 2$ and $a = 1$, $b = c = 0$, then let $T_6$ be the tree obtained from $K_4(a,b,c,d)$ by deleting $d-1$ pendant edges incident to $v_4$ and attaching these $d-1$ pendant edges to the unique leaf adjacent to $v_4$.
	Clearly, $T_6 \in \mathcal{G}_{n,4}$.
	By Lemma \ref{lem-2-1}, $\rho(T_6) < \rho(K_4(a,b,c,d))$, contradicting the minimality of $\rho(T_4^*)$.
	
	If $d = 1$ and exactly one of $v_1, v_2, v_3$ has degree greater than $3$.
	Let $v'$ be that vertex.
	Let $T_7$ be the tree obtained from $K_4(a,b,c,d)$ by deleting all pendant edges incident to $v'$ and then attaching the same number of pendant edges to the leaf adjacent to $v_4$.
	Clearly, $T_7 \in \mathcal{G}_{n,4}$.
	By Lemma \ref{lem-2-1}, $\rho(T_7) < \rho(K_4(a,b,c,d))$, contradicting the minimality of $\rho(T_4^*)$.

	If $d = 1$ and at least two of $v_1, v_2, v_3$ have degree greater than $3$.
	Let $v''$ and $v'''$ be two such vertices.
	Let $T_8$ be the tree obtained from $K_4(a,b,c,d)$ by subdividing an edge on the path $P_{v''v'''}$ and then deleting one pendant edge incident to $v_4$.
	Clearly, $T_8 \in \mathcal{G}_{n,4}$.
	By Lemmas \ref{subdivision} and \ref{propersubgraph}, $\rho(T_8) < \rho(K_4(a,b,c,d))$, contradicting the minimality of $\rho(T_4^*)$.
\end{proof}

\begin{lem}\label{lem-3.12}
	$T_4^* \ncong K_5(a,b,c,d)$ for any integers $a \ge 1$, $b \ge 0$, $c \ge 0$, and $d \ge 1$.
\end{lem}

\begin{proof}
	Assume for contradiction that $T_4^* \cong K_5(a,b,c,d)$ with $a \ge 1$, $b \ge 0$, $c \ge 0$, and $d \ge 1$.
	Observe that when both $b$ and $c$ of $K_{5}(a,b,c,d)$ are zero, it can be regarded as $K_{4}(a,b,c,d)$.
	By Lemma \ref{lem-3.11}, the conclusion is immediate.
	Therefore, it suffices to consider $K_{5}(a,b,c,d)$ with $b \ge 1$, without loss of generality. 
	
	Let $\mathbf{x}=(x_1, x_2, \dots, x_n)$ be the Perron vector of $K_{6}(a,b-1,c,d)$. 
	Let $u_1$ denote the unique common neighbor of $v_1$ and $v_2$, and let $u_2$ denote the unique common neighbor of $v_2$ and $v_3$, as shown in Fig \ref{fig-main-tree-4}. Let $W_i$ denote the set of all leaves adjacent to $v_i$, for $i=1,2,3,4$.
	
If $x_{v_2} \ge x_{u_2}$, by deleting the edge $u_2v_3$ from $K_{6}(a,b-1,c,d)$ and adding the edge $v_2v_3$, we obtain $K_{5}(a,b,c,d)$. By Lemma~\ref{lem-perron-entry-radius}, we have $\rho(K_{6}(a,b-1,c,d)) < \rho(K_{5}(a,b,c,d))$, contradicting the minimality of $\rho(T_4^*)$.  If $x_{v_2} < x_{u_2}$, then $K_{5}(a,b,c,d)=K_{6}(a,b-1,c,d)-\{v_2w:w \in W_2\}-u_1v_2+\{u_2w:w \in W_2\}+u_1u_2$. By Lemma~\ref{lem-perron-entry-radius}, we again have $\rho(K_{6}(a,b-1,c,d)) < \rho(K_{5}(a,b,c,d))$, which also contradicts the minimality of $\rho(T_4^*)$.
\end{proof}

\begin{lem}\label{lem-3.13}
	$T_4^* \not\cong K_7(a,b,c,d)$ for any integers $a,b,c,d \ge 1$.
\end{lem}

\begin{proof}
	Assume for contradiction that $T_4^* \cong K_7(a,b,c,d)$ with $a,b,c,d \ge 1$.
	
	If At least one of $a,c,d$ is at least $2$.
	Without loss of generality, suppose $a \ge 2$.
	Let $T_{9}$ be the tree obtained from $K_7(a,b,c,d)$ by subdividing the edge $v_1v_2$, then deleting one pendant edge incident to $v^*$.
	Clearly, $T_{9} \in \mathcal{G}_{n,4}$.
	By Lemmas \ref{subdivision} and \ref{propersubgraph}, we have $\rho(T_{9}) < \rho(K_7(a,b,c,d))$, contradicting the minimality of $\rho(T_4^*)$.
	
	If $a = c = d = 1$.
	Let $T_{10}$ be the tree obtained from $K_7(a,b,c,d)$ by deleting the edge $v_2v_4$ and connecting $v_4$ to the unique leaf adjacent to $v_3$.
	Clearly, $T_{10} \in \mathcal{G}_{n,4}$.
	By Lemma \ref{lem-2-1}, $\rho(T_{10}) < \rho(K_7(a,b,c,d))$, contradicting the minimality of $\rho(T_4^*)$.
	
	If $a\geq 2$ or $c\geq 2$ or $d\geq 2$, we notice that $K_7(a,b,c,d)$ has an internal edge that when subdivided keeps the tree in $\mathcal{G}_{n,4}$ and by the same reasoning it can't be minimal. It remains to study the case $a=c=d=1$ and $b\geq 1$, when there is no internal path. As $n\geq 10$ already, we may assume that $b\geq 2$. In this case we take one pendant edge in $v_2$ and make it pendant in the pendant edge of $v_3$. The resulting tree is still in $\mathcal{G}_{n,4}$ and has smaller spectral radius by Lemma~\ref{lem-2-1}, contradicting the minimality. 
\end{proof}

\begin{lem}\label{lem-3.14}
	$T_4^* \not\cong K_8(a,b,c,d)$ for any integers $a,b,c,d \ge 1$.
\end{lem}

\begin{proof}
	Assume for contradiction that $T_4^* \cong K_8(a,b,c,d)$ where $a,b,c,d \ge 1$.
	
	If $c \ge 2$ or $d \ge 2$. We may assume $c\geq 2$.
	Let $T_{11}$ be the tree constructed from $K_8(a,b,c,d)$ by subdividing the edge $v_2v_3$, and deleting one pendant edge incident to $v^*$.
	Clearly, $T_{11} \in \mathcal{G}_{n,4}$.
	By Lemmas \ref{subdivision} and \ref{propersubgraph}, we have $\rho(T_{11}) < \rho(K_8(a,b,c,d))$, which contradicts the minimality of $\rho(T_4^*)$.
	
	If $c = d = 1$.
	Let $T_{12}$ be the tree obtained from $K_8(a,b,c,d)$ by deleting the edge $v_2v_4$ and connecting $v_4$ to the unique leaf adjacent to $v_3$.
	Clearly, $T_{12} \in \mathcal{G}_{n,4}$.
	By Lemma \ref{lem-2-1}, $\rho(T_{12}) < \rho(K_8(a,b,c,d))$, contradicting the minimality of $\rho(T_4^*)$.
\end{proof}

\begin{lem}\label{claim-3.15}
	$T_4^* \ncong K_9(a,b,c,d)$ for any integers $a \ge 1$, $b \ge 0$, $c \ge 1$, and $d \ge 1$.
\end{lem}

\begin{proof}
	Assume for contradiction that $T_4^* \cong K_9(a,b,c,d)$ where $a \ge 1$, $b \ge 0$, $c \ge 1$, and $d \ge 1$. We analyze the following cases:
	
	If $d \ge 2$.
	Let $T_{13}$ be the tree obtained from $K_9(a,b,c,d)$ by first subdividing the edge $v_2v_4$, then deleting one pendant edge incident to $v^*$.
	Clearly, $T_{13} \in \mathcal{T}_{n,4}$.
	By Lemmas \ref{subdivision} and \ref{propersubgraph}, we have $\rho(T_{13}) < \rho(K_9(a,b,c,d))$, which contradicts the minimality of $\rho(T_4^*)$.
	
	If $d = 1$ and either $a \ge 2$ or $c \ge 2$.
	Without loss of generality, suppose $a \ge 2$.
	Let $T_{14}$ be the tree obtained from $K_9(a,b,c,d)$ by first subdividing an edge on the path $P_{v_1v_2}$, then deleting one pendant edge incident to $v_4$.
	Clearly, $T_{14} \in \mathcal{T}_{n,4}$.
	By Lemmas \ref{subdivision} and \ref{propersubgraph}, $\rho(T_{14}) < \rho(K_9(a,b,c,d))$, which contradicts the minimality of $\rho(T_4^*)$.
	
	If $d = 1$, $a = 1$, and $c = 1$.
	Let $T_{15}$ be the tree obtained from $K_9(a,b,c,d)$ by deleting the edge $v_2v_4$ and connecting $v_4$ to the unique leaf adjacent to $v_3$.
	Clearly, $T_{15} \in \mathcal{T}_{n,4}$.
	By Lemma \ref{lem-2-1}, $\rho(T_{15}) < \rho(K_9(a,b,c,d))$, which contradicts the minimality of $\rho(T_4^*)$.
\end{proof}

\begin{lem}\label{lem-3.16}
	$T_4^* \ncong K_{11}(a,b,c,d)$ for any integers $a,b,c,d \ge 1$.
\end{lem}
\begin{proof} Assume, by way of contradiction, that $T_4^* \cong K_{11}(a,b,c,d)$ with $a,b,c,d \ge 1$.
	Observe that when $d = 1$ in $K_{11}(a,b,c,d)$, it can be regarded as $K_{12}(a,b,c,d-1)$.
	By Lemma~\ref{notK12K13}, the result follows directly.
	Hence, we only need to consider $K_{11}(a,b,c,d)$ with $d \ge 2$.
	
	Let $\mathbf{y}=(y_1, y_2, \dots, y_n)$ be the Perron vector of $K_{12}(a,b,c,d-1)$. 
	Let $u_2$ denote the unique common neighbor of $v_3$ and $v_4$, as shown in Fig \ref{fig-main-tree-6}. Let $W_i$ denote the set of all leaves adjacent to $v_i$, for $i=1,\ldots,4$.
	
If $y_{v_4} \ge y_{u_2}$, by deleting the edge $v_3u_2$ from $K_{12}(a,b,c,d-1)$
	and adding the edge $v_3v_4$, we obtain $K_{11}(a,b,c,d)$.
	By Lemma~\ref{lem-perron-entry-radius}, we have $\rho(K_{12}(a,b-1,c,d)) < \rho(K_{11}(a,b,c,d))$,
	contradicting the minimality of $\rho(T_4^*)$. Otherwise, if $y_{v_4} < y_{u_2}$,
	then $K_{11}(a,b,c,d)=K_{12}(a,b,c,d-1)-\{v_4w:w \in W_4\}+\{u_2w:w \in W_4\}$. By Lemma~\ref{lem-perron-entry-radius}, we again have $\rho(K_{12}(a,b,c,d-1)) < \rho(K_{11}(a,b,c,d))$,
	which also contradicts the minimality of $\rho(T_4^*)$.
\end{proof}

\begin{lem}\label{notK12K13}
	\(T_4^*\) is not isomorphic to neither \(K_{12}(a,b,c,d)\) nor \(K_{13}(a,b,c,d)\).
\end{lem}

\begin{proof}
	Observe that $a$ and $b$ of $K_{12}(a,b,c,d)$ cannot be zero simultaneously. Without loss of generality, assume that $b \ge 1$. The characteristic polynomial of a star \(K_{1,n}\) is given by \(\varphi(K_{1,n}, x) = x^{n-1}(x^2 - n)\). Both \(K_{12}(a,b,c,d)\) and \(K_6(a,b-1,c,d)\) are composed of several star graphs. 
	By Lemma~\ref{characteristic}, we compute the characteristic polynomials as follows:
	$
	\varphi(K_{12}(a,b,c,d), x) = x^{a+b+c+d-2}\big[(x^4 - (a+b+2)x^2 + ab + a + b)(x^4 - (c+d+2)x^2 + cd + c + d) 
	- (x^2 - a)(x^2 - b)(x^2 - d - 1)\big]
	$, and 
	$\varphi(K_6(a,b-1,c,d), x) = x^{a+b+c+d-2}\big[ (x^4 - (a+b+1)x^2 + a(b-1) + a + b - 1) \cdot (x^4 - (c+d+3)x^2 + (c+1)d + c + d + 1)
	- (x^4 - (c+d+2)x^2 + cd + c + d)(x^2 - (a+1))\big].
	$
	
	For \(x > \max\left\{\sqrt{d+1}, \sqrt{b-1}\right\}\), we obtain:
	\(
	\varphi(K_6(a,b-1,c,d), x) - \varphi(K_{12}(b,a,c,d), x)\\ = x^{a+b+c+d-2}(x^2 - (d+1))(x^2 - (b-1)) > 0.
	\)
	
	Notice that $\rho(K_{1,d+1})=\sqrt{d+1}$ and $\rho(K_{1,b-1})= \sqrt{b-1}$.
	Since both \(K_6(a,b-1,c,d)\) and \(K_{12}(b,a,c,d)\) contain \(K_{1,d+1}\) and \(K_{1,b-1}\) as induced proper subgraphs, it follows from Lemma~\ref{propersubgraph} that:
	\(
	\rho(K_6(a,b-1,c,d)), \rho(K_{12}(b,a,c,d) > \max\big\{ \sqrt{d+1}, \sqrt{b-1}\big\}.
	\)
	Hence, we conclude that:
	\(
	\rho(K_6(a,b-1,c,d)) < \rho(K_{12}(b,a,c,d)),
	\)
	which implies that \(T_4^*\ncong K_{12}(a,b,c,d)\).
	
	Observe that among $a, b, c,$ and $d$ of $K_{13}(a,b,c,d)$, at most one can be zero. Without loss of generality, assume that $c \ge 1$. Next, we compare the characteristic polynomials of $K_{12}(a,b,c-1,d)$ and $K_{13}(a,b,c,d)$.
	By Lemma~\ref{characteristic}, they are given by
	$
	\varphi(K_{13}(a,b,c,d), x) \\= x^{a+b+c+d-3}
	(x^2 - a)\big[ (x^2 - b)(x^4 - (c+d+2)x^2 + cd + c + d) - (x^2 - c)(x^2 - d) - (x^2 - b)(x^2 - c)(x^2 - d) \big]
	$
	, and 
	\(
	\varphi(K_{12}(a,b,c-1,d), x) = x^{a+b+c+d-3} \big[(x^4 - (a+b+2)x^2 + ab + a + b)\cdot (x^4 - (c+d+1)x^2 + (c-1)d + c + d - 1) -
	(x^2 - a)(x^2 - b)(x^2 - d - 1) \big].
	\)

	For \(x > \max\big\{\sqrt{c-1}), \sqrt{a}, \sqrt{b}\big\}\), we have:
	\[
	\varphi(K_{12}(a,b,c-1,d), x) - \varphi(K_{13}(a,b,c,d), x) = x^{a+b+c+d-3}(x^2 - (c-1))(2x^2 - a - b) > 0.
	\]
	
	Since both graphs contain \(K_{1,c-1}\), \(K_{1,a}\), and \(K_{1,b}\) as induced proper subgraphs, from Lemma~\ref{propersubgraph} we have
	\(
	\rho(K_{12}(a,b,c-1,d)),\rho(K_{13}(a,b,c,d))>\max\big\{\sqrt{c-1}), \sqrt{a}, \sqrt{b}\big\}.
	\)
	Hence,
	\(
	\rho(K_{13}(a,b,c,d)) > \rho(K_{12}(a,b,c-1,d)),
	\)
	which implies \(T_4^*\ncong K_{13}(a,b,c,d)\).
\end{proof}

From Lemma~\ref{lem-3.8} to Lemma~\ref{notK12K13}, one may conclude that we have discarded all the candidates $K_i(a,b,c,d)$ for $i \ne 6$ and $i \ne 10$, and obtain the following proposition.

\begin{pro}\label{prop:remainsK6K10} Let  $T_4^* $ be the minimizer in $\mathcal{G}_{n,4} $, for $n\geq 11$. Then $T_4^*\cong K_6(a,b,c,d)$ or $T_4^*\cong K_{10}(a,b,c,d)$.
\end{pro}

\color{black}

By the proof of the main theorem in \cite{LouGuo2022}, we can derive the following result.
\begin{lem}[\cite{LouGuo2022}]\label{lem-3.6}
Let \( n \geq 19, T' \) be the extremal tree with minimum spectral radius among all trees in the set
\[
\Big\{ K_6(a,b,c,d) \;|\;  a+b+c+d = n-7 \Big\}
\cup
\Big\{ K_{10}(s_1,s_2,s_3,s_4) \;|\; \textstyle\sum_{i=1}^4 s_i = n-7 \Big\},
\]
where the trees \( K_6(a,b,c,d) \) and \( K_{10}(s_1,s_2,s_3,s_4) \) are depicted in Fig.\ref{fig-1.1}.
Let \( s = \lfloor \frac{n-7}{4} \rfloor \). Then we have
\[
	T' \cong 
	\begin{cases} 
		\begin{cases} K_6(s + 1, s - 1, s - 1, s + 1) \text{ and} \\ K_{10}(s + 1, s - 3, s + 1, s + 1) \end{cases} & \text{if } n - 7 = 4s, \\
		K_{10}(s + 1, s - 2, s + 1, s + 1) & \text{if } n - 7 = 4s + 1, \\
		\begin{cases} K_6(s + 1, s, s, s + 1), \\ K_6(s + 2, s - 1, s, s + 1) \text{ and} \\ K_6(s + 2, s - 1, s - 1, s + 2) \end{cases} & \text{if } n - 7 = 4s + 2, \\
		K_{10}(s + 2, s - 3, s + 2, s + 2) & \text{if } n - 7 = 4s + 3.
	\end{cases}
	\]	
\end{lem}

\begin{proof}[\bf{Proof of Theorem \ref{thm-1.3}}]
By Proposition~\ref{prop:remainsK6K10},  we conclude that 
\[
T_4^*\in \Big\{ K_6(a,b,c,d) \;|\;  a+b+c+d = n-7 \Big\}
\cup
\Big\{ K_{10}(s_1,s_2,s_3,s_4) \;|\; \textstyle\sum_{i=1}^4 s_i = n-7 \Big\}
\]
By Lemma \ref{lem-3.6}, the conclusion of Theorem \ref{thm-1.3} follows immediately.
\end{proof}

\section*{Acknowledgements}
Supported by the National Natural Science Foundation of China (Nos. 12271362, 
\! 11671344) and the Natural Science Foundation of Shanghai (No.25ZR1402390). 
V. Trevisan also acknowledges the support of CNPq grants 408180/2023-4 and 308774/2025-6.

\end{document}